\documentclass{amsart}
\pdfoutput=1
\usepackage[utf8]{inputenc}

\usepackage{cite}

\usepackage{amsmath}
\usepackage{amssymb}
\usepackage{amsfonts}
\usepackage{amsthm}
\usepackage{geometry}
\usepackage{enumerate}
\usepackage{tikz}
\usepackage{wrapfig}
\usepackage{tkz-euclide}

\usepackage{hyperref}
\usepackage{geometry,amsmath,amssymb,bbm}

\usepackage{todonotes}

\usepackage{xfrac}

\newcommand{\R}{\begin{ensuremath}	\mathbb R\end{ensuremath}} 
\newcommand{\N}{\begin{ensuremath}	\mathbb N\end{ensuremath}}

\newcommand{\Q}{\begin{ensuremath}	\mathbb Q\end{ensuremath}}

\renewcommand{\d}{\begin{ensuremath}\,\mathrm d\end{ensuremath}} 
\newcommand{\del}{\partial}

\newcommand{\alpham}{\alpha_\ominus}
\newcommand{\alphap}{\alpha_\oplus}

\newcommand{\eps}{\varepsilon}

\newtheorem{theorem}{Theorem}
\newtheorem{lemma}[theorem]{Lemma}

\newtheorem{example}[theorem]{Example}
\newtheorem{remark}[theorem]{Remark}
\newtheorem{corollary}[theorem]{Corollary}

\newtheorem{question}{Question}
\newtheorem*{requirement}{Requirement}

\title{Pattern size in Gaussian fields from spinodal decomposition }
\author[L.~A.~Bianchi]{Luigi Amedeo Bianchi}
    \address{L.~A.~Bianchi, Institut f\"ur Mathematik\\ Universit\"at Augsburg\\ D-86135 Augsburg, Germany}
    \email{\href{mailto:luigi.bianchi@math.uni-augsburg.de}{luigi.bianchi@math.uni-augsburg.de}}
    \urladdr{\url{http://www.math.uni-augsburg.de/prof/ana/arbeitsgruppe/bianchi/}}
    \author[D.~Bl\"omker]{Dirk Bl\"omker}
    \address{D.~Bl\"omker, Institut f\"ur Mathematik\\ Universit\"at Augsburg\\ D-86135 Augsburg, Germany}
    \email{\href{mailto:dirk.bloemker@math.uni-augsburg.de}{dirk.bloemker@math.uni-augsburg.de}}
    \urladdr{\url{http://www.math.uni-augsburg.de/prof/ana/arbeitsgruppe/bloemker/}}
    \author[P.~Wacker]{Philipp Wacker}
    \address{P.~Wacker, Institut f\"ur Mathematik\\ Universit\"at Augsburg\\ D-86135 Augsburg, Germany}
    \email{\href{mailto:dueren@math.uni-augsburg.de}{dueren@math.uni-augsburg.de}}
    \urladdr{\url{http://www.math.uni-augsburg.de/prof/ana/arbeitsgruppe/dueren/}}
  \thanks{P.W. was supported by a Cusanuswerk scholarship. L.A.B. was supported by the German Science Foundation (DFG),
  grant number BL 535/9-2
  ``Mehrskalenanalyse stochastischer partieller Differentialgleichungen
  (SPDEs)''.}
\date{\today}

\begin{document}
\begin{abstract}
We study the two-dimensional snake-like pattern that arises in 
phase separation of alloys described by spinodal decomposition in the Cahn-Hilliard model.
These are somewhat universal patterns due to an overlay of eigenfunctions
of the Laplacian with a similar wave-number.
Similar structures appear in other models like reaction-diffusion systems describing animal 
coats' patterns or vegetation patterns in desertification.

Our main result studies random functions given by cosine Fourier series with independent Gaussian coefficients, 
that dominate the dynamics in the Cahn-Hilliard model.
This is  not a cosine process, as the sum is 
taken over domains in Fourier space that not only
grow and scale with a parameter of order $1/\varepsilon$,
but also move to infinity. 
Moreover, the model under consideration is neither stationary nor isotropic.

To study the pattern size of nodal domains 
we consider the density of zeros
on any straight line through the spatial domain.
Using a theorem by Edelman and Kostlan and weighted ergodic 
theorems that ensure the convergence of the moving sums, 
we show that the average distance of zeros is asymptotically of order $\varepsilon$ 
with a precisely given constant.
\end{abstract}
\maketitle

\section{Introduction}

We are interested in studying the patterns that form in the solution of the stochastic Cahn-Hilliard equation
during the separation process called spinodal decomposition. 
This equation, originally introduced in~\cite{cahn} and~\cite{cahnhilliard}, 
models the relative concentration of two components in an alloy after quenching an initially homogeneous mixture.

Similar structures appear in many other situations.
One example are the patterns in coats  
of animals like zebras or tigers \cite{Mu:MathBio},
where the underlying system is a reaction-diffusion system with a Turing instability, as argued by Sander and Wanner in~\cite{sanwan03}. 
Another occurrence are the vegetation patterns in desertification like the one observed for the tiger bush in Africa, as studied for example by Siero, Siteur et al. in~\cite{Rademach1,Rademach2}. In those papers they provide a mathematical study of the emergence of 
pattern on an unbounded domain in the extended 
Klausmeier model, which is similar to a reaction-diffusion system.

More similar to the Cahn-Hilliard setting is the Swift-Hohenberg equation (see for example the review~\cite{CrHo93}). 
A result very much related to the result presented here is given by Bl\"omker and Maier-Paape in~\cite{Bl-MP:03}. 
The main difference is that significantly fewer Fourier-modes determine the pattern.
The initial stage of hill formation of an example from surface growth \cite{RaLiHa01,HoNa02} 
also used in ion sputtering \cite{sput-2}
exhibits similar patterns, too. For details see the review article~\cite{Bl-Ro:15}.

We believe that in all these examples the characteristic snake like pattern that 
appear initially is mainly 
due to random composition of eigenfunctions of the Laplacian of similar wave-length,
where the solution is described by a high dimensional strongly unstable space 
and the nonlinearity does not yet play a significant role.
In order to outline this idea, we focus in the following sections 
on the stochastic Cahn-Hilliard equation on the square.

The main question that we want to answer rigorously in this paper is the following:
\begin{center}
 \emph{What is the characteristic thickness of the pattern (i.e. the snake-like structures)?}
\end{center}
To address this, we count the (average) number of zeros on any  straight line across
the domain, using a result given by Edelman and Kostlan in~\cite{edelman1995many}.

In the context of spinodal decomposition, a
partial result was obtained by Maier-Paape and Wanner in~\cite{maier1998spinodal}, where they tried to estimate the size of balls
that would fit into the nodal domain. A numerical computation of Betti numbers was performed in a series of papers
\cite{gammiswan05,daykalmiswan07,daykalwan09}.

Moreover the structure of nodal domains for Gaussian random fields 
is the topic of numerous recent publications, which we do not try to survey in full detail.
Under the assumption of stationarity and isotropy there are asymptotic results 
on the numer of zeros along lines. See for example  \cite{dennis2007nodal,longuet1957statistical}.
 Minkowski-functionals are also treated (see \cite{AdlerLNM} for a review), or \cite{SodinLN} 
for references to asymptotic results of the number of connected components.
We comment on these approaches later on in more detail.

\subsection{The Cahn-Hilliard-Cook equation}

The stochastic Cahn-Hilliard (or Cahn-Hilliard-Cook) equation 
was introduced by Cook in~\cite{cook} as a stochastic modification of the originally deterministic model.
It can be written as follows:
\[
\del_t u = -\Delta(\eps^2\Delta u + f(u)) + \del_t W, 
\]
where the noise $\del_t W$ is the derivative of a $Q$-Wiener process 
and $f$ is the derivative of a double well potential, 
where a standard choice  is $f(u) = u-u^3$, although the true nonlinearity 
introduced by Cahn and Hilliard should exhibit logarithmic poles. 
We consider it on the square domain $[0,1]^2$, 
with Neumann boundary conditions $\del_\nu u = \del_nu \Delta u$. 
In the physical model of alloys, $u$ models the rescaled concentration of one component,
and the extreme values $u=\pm 1$ correspond to $0\%$ and $100\%$ concentration 
of the first component in any point.

The canonical initial condition for the phase separation in spinodal decomposition 
is a homogeneous concentration, 
constant on the whole domain. Due to the presence of the noise, 
after some time the homogeneous picture is broken, decomposition starts playing its role and snake-like pattern appear and persist, 
giving place to situations like the one simulated in Figure~\ref{im:spinodal}.

\begin{figure}[hbtp]
\centering
\includegraphics[width=0.5\textwidth]{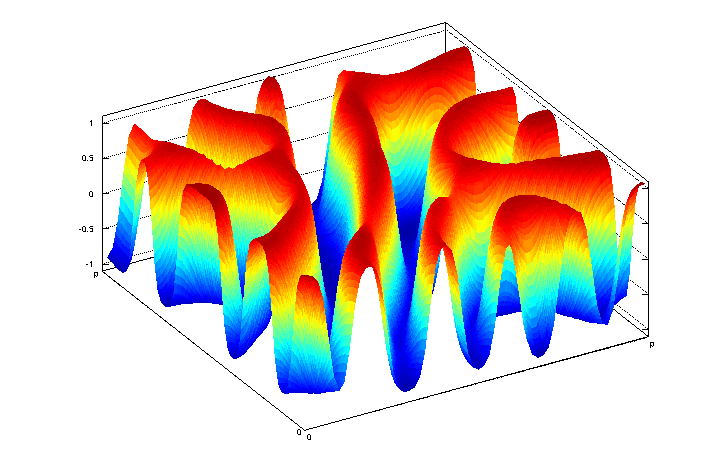}
\caption{Pattern in Spinodal decomposition. The  snake-like pattern appears
as the nodal domains in the level set of the average concentration $0$.}
\label{im:spinodal}
\end{figure}

For the Cahn-Hilliard equation, one can consider the set of the most unstable eigenvalues, 
which dominate the dynamics for a long time,  also called ``strongly unstable subspace''.
This was  studied in the deterministic model by Maier-Paape and Wanner in~\cite{maier1998spinodal, MPW2}
and later in the stochastic setting by Bl\"omker, Maier-Paape and Wanner in~\cite{bloemker_maierpaape_wanner, DBMPW1} (see also the review \cite{SpiDecRev}).
Later Sander and Wanner 
extended in~\cite{SanderWanner} the approximation result by linearisation to unexpectedly large radii.

It is interesting to notice that the evolution of pattern complexity in the spinodal decomposition
seems to be quite different in the deterministic model with random initial conditions or the  
stochastic model with constant initial condition.
In~\cite{gammiswan05} it was shown numerically that the deterministic model 
exhibits an unnatural increase of complexity during the separation process.

\subsection{Linearised Cahn-Hilliard-Cook equation}

While the Cahn-Hilliard equation is highly nonlinear, 
its dynamics is reasonably well approximated by linearisation in the first phase up to unexpectedly large radii. 
So let us discuss the linearised system (around the initial conditions $u=m$) to motivate why  we study random cosine series later.

Such linearised system can be written as
\begin{equation}\label{eq:linearized}
	\del_t u = Au + \del_t W\;, \qquad u(0)=m,
\end{equation}
where $A = -\eps^2 \Delta^2 - \Delta f^\prime(m)$ is a self-adjoint linear operator having a complete  orthonormal 
system of eigenfunctions. 
In our study we focus on the simple domain $[0,1]^2$ to avoid additional pattern complexity introduced by a
complicated shape of the boundary. The $L^2$-basis on the square is made of cosine functions  
\[
	e_{k,l}(x) = C\cos(k\pi x)\cos(l\pi y),
\]
with $A e_{k,l} = \lambda_{k,l} e_{k,l}$, where
\[ 
	\lambda_{k,l}=-\eps^2 (k^2+l^2)^2\pi^4+(k^2+l^2)\pi^2f^\prime(m).
\]

The solution to~\eqref{eq:linearized} is given by the stochastic convolution (see \cite{daPrato}),
which simplifies here to a cosine series with random coefficients.
\[
	W_A(t) =\sum_{k,l\in\N} \int_0^t e^{(t-s)A} \d W(s) 
	= \sum_{k,l\in\N}\alpha_{k,l} \int_0^t e^{(t-s)\lambda_{k,l}} \d B_{k,l}(s) e_{k,l},
\]
where the $B_k$ are independent Brownian motions and the $Q$-Wiener process $W$ has a joint eigenbasis with $A$ such that $Qe_{k,l}=\alpha_{k,l}^2 e_{k,l}$.
This is a usual but quite strong assumption, that $Q$ commutes with $A$. As this is only our motivating example,
for simplicity of presentation we do not enter this discussion here. Some details and further references can be found in 
\cite{SpiDecRev}. One also could think of space-time white noise, where $Q$ is the identity and thus all $\alpha_k$ are one. 

The strongly unstable subspace also called ``strong subspace'', for short, is defined 
in Fourier space by the wave-vectors
\[
	R_\eps^\gamma := \{(k,l)\in \N^2: \lambda_{k,l} > \gamma \lambda_{\max}\},
\] 
for
some $\gamma\in (0,1)$ close to $1$. In particular, 
in the two-dimensional setting we consider, 
the set $R_\eps^\gamma$ is a quarter-ring in the $\N^2$ lattice, 
containing $\mathcal{O}(\eps^{-2})$ many wave-vectors corresponding to 
eigenvalues close to the maximum.
In the following we do not always specify the explicit 
dependence of $R_\eps^\gamma$ on $\gamma$ and write $R_\eps$ for short.
Note that this set not only is growing in size for $\eps\to0$,
but also moves as a whole towards infinity.

As already outlined, as long as the solution is not too large, 
the dynamics of the nonlinear Cahn-Hilliard equation is dominated 
by the projection $P_{R_\eps}$ of the stochastic convolution on the strong subspace.
By this we mean the restriction of the Fourier series to wave-vectors in $R_\eps$, which is given by
\[
	P_{R_\eps} W_A(t) = \sum_{m = (k,l)\in{R_\eps}} \alpha_m c_m e_m,
	\qquad \text{with}\quad  c_m=\int_0^t e^{(t-s)\lambda_m} \d B_m(s)\;.
\]
The random variables in the family $\left\{c_m\right\}_{m\in R_\eps}$ are by definition independent centred Gaussians.
By It\=o-isometry the variance of $c_m$ is
\[
	 \mathbb{E}c_m^2= \int_0^t e^{(t-s)2\lambda_m}\d s 
	 = \frac1{2\lambda_m}(1-e^{-2\lambda_m t}) 
	 \approx  \frac1{2\lambda_m}
	 =\mathcal{O}(\eps^2), 
\]
for times $t\approx\eps^2$, which is close to the time-scale on 
which the first phase of spinodal decomposition was described (see~\cite{DBMPW1}), 
recalling that $\lambda_m = \mathcal{O}(\eps^{-2})$, as it is of order of the largest eigenvalue. 
Hence after linearisation and projection 
via $P_{R_\eps}$, the solution at a fixed time $t$ seems to be well approximated by
\begin{equation}\label{eq:randomFourier}
	u(x,y) \approx \sum_{(k,l)\in R_\eps} c_{k,l} \cdot \cos(k\pi x)\cos(l\pi y),
\end{equation}
where the $c_{k,l}$ are independent centred Gaussians with similar variances.
This motivates the choice of our toy model, which we will present in the next subsection, fixing for ease of presentation all the variances of the coefficients to be the same.

\subsection{Random Fourier Series}

For the main part of the paper we consider the random function
\begin{equation}\label{eq:fnc}
	f(x,y) = \sum_{(k,l)\in R_\eps} c_{k,l} \cdot \cos(k\pi x)\cos(l\pi y),
\end{equation}
on the unit square $x,y\in[0,1]^2$, with the random coefficients $c_{k,l}$ 
being independent and identically distributed centred Gaussian random variables.
Later we discuss also the impact of other domains in Fourier space, and not only the ring.

As we showed in the previous subsection, this is  a simplified version 
(adding the identically distributed assumption) 
of the approximation~\eqref{eq:randomFourier} of the stochastic Cahn-Hilliard equation.
We furthermore set the variance to be $1$, i.e. $c_{k,l}\sim N(0,1)$.
Doing this, we ignore both the time-dependence and the inhomogeneity 
in the growth rates of the strong Fourier modes for the sake of simplicity.

We define the subset $R_\eps$ of strong modes, introduced above in the Cahn-Hilliard case, as
\[
	{R_\eps} = \{(k,l)\in\N^2|~ \alpham < \sqrt{(k\eps)^2+(l\eps)^2} <  \alphap\},
\] 
where the parameters are 
\[
	\alphap = \sqrt{\frac{1+\sqrt{1-\gamma}}{2\pi^2}} 
	\quad\text{and}\quad 
	\alpham = \sqrt{\frac{1-\sqrt{1-\gamma}}{2\pi^2}}
	\quad\text{with}\quad \gamma\in (0,1).
\]
Although the model \eqref{eq:fnc} is somewhat reminiscent of the cosine-process (see e.g.~\cite{AdlerTaylorGeo}) 
or the random wave model (see e.g.~\cite{dennis2007nodal}), it is fundamentally different in the sense that 
it is neither stationary nor isotropic. The law  of the random function $f$
might change under translation or rotation
(as a function extended to $\mathbb{R}^2$). 
It is an easy calculation to show that $f(x)$ is a centered Gaussian  with covariance
\[
\mathbb{E}f(x)f(y) = q(x+y)+q(x-y) 
\quad\text{for} \quad
q(z)=\frac12 \sum_{(k,l)\in R_\eps} \cos(k\pi z) \cos(l\pi z)\;.
\]

\subsection{Pattern size}

\begin{figure}[hbtp]
\centering
\includegraphics[width=.45\textwidth]{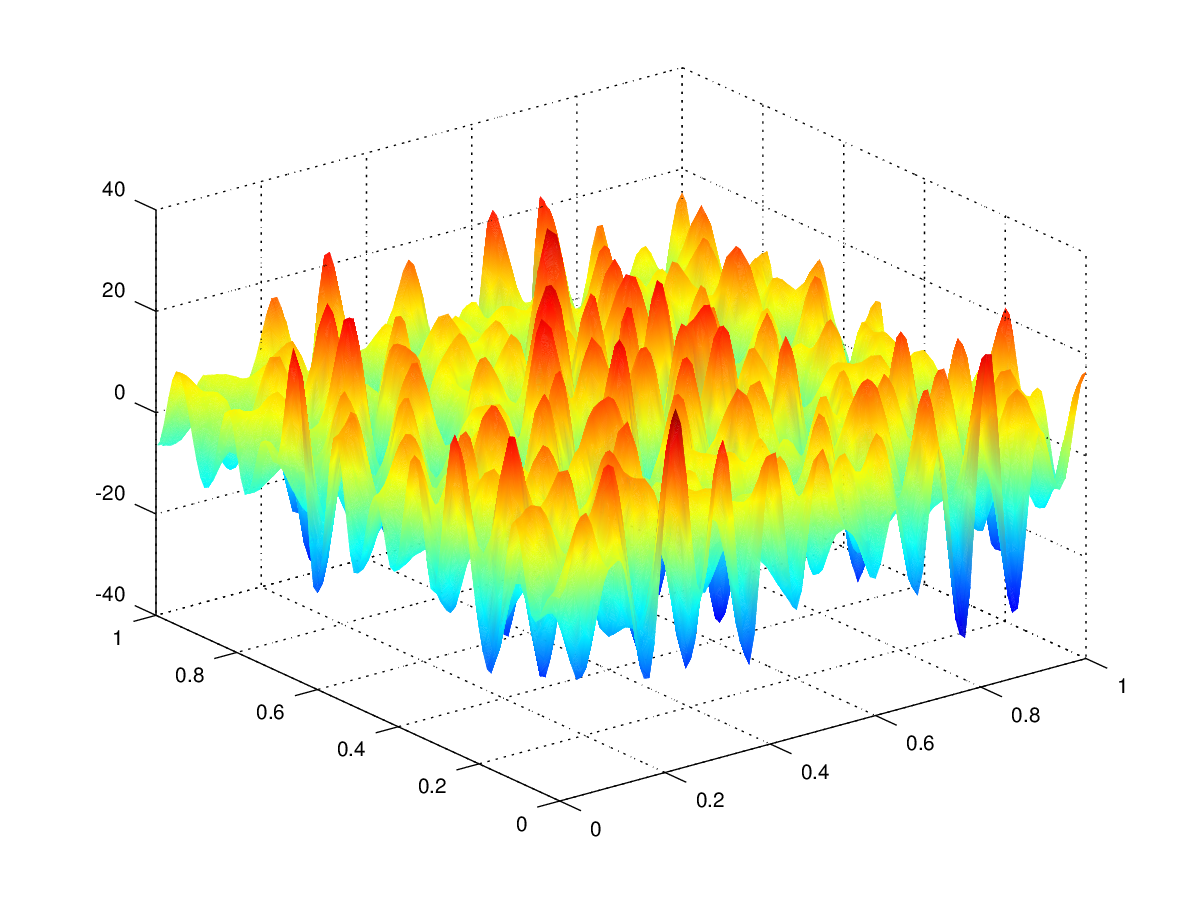}
\includegraphics[width=.45\textwidth]{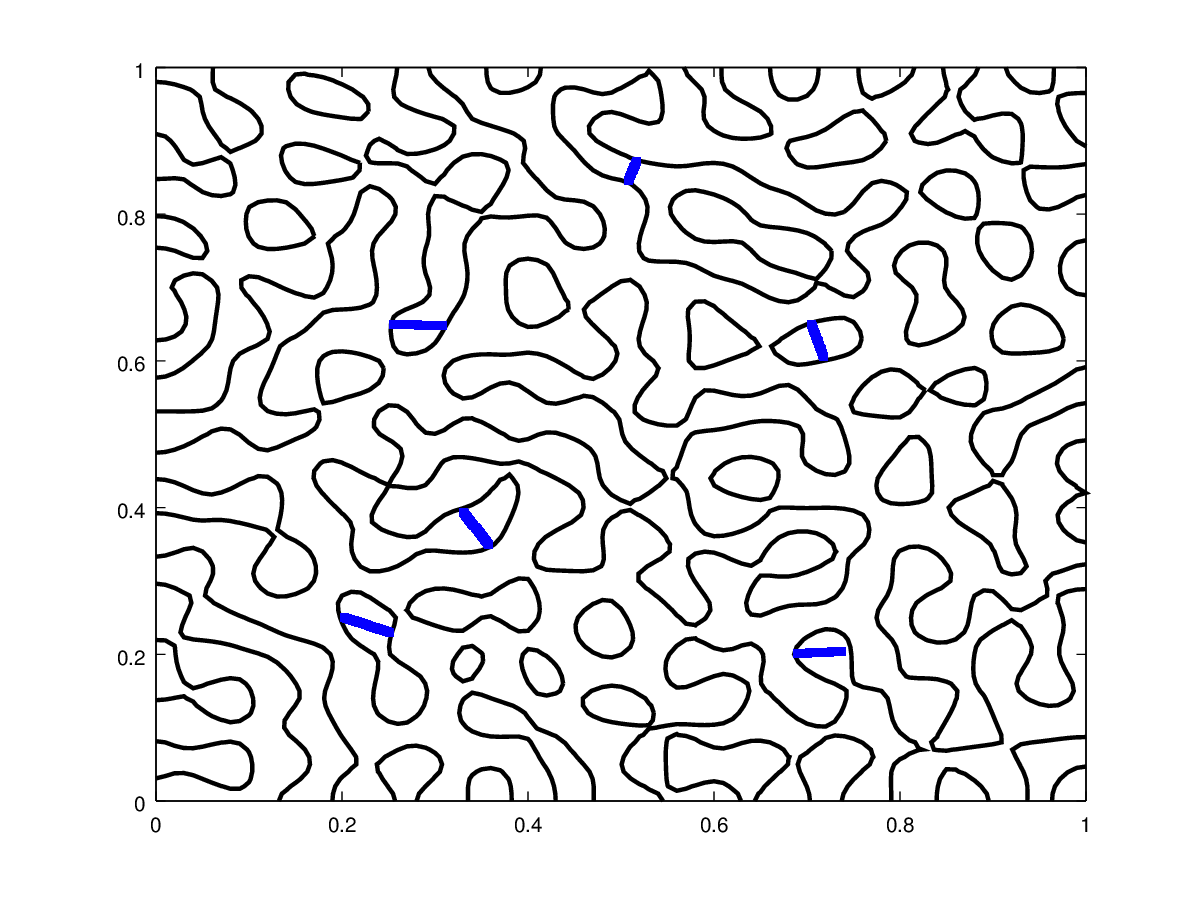}
\caption{A simulation of the function $f$ for $\gamma = 0.8$  with a plot of its zero-level set.
On the latter we picked some measurements of the pattern thickness. 
Empirically, the average distance is $2\pi\eps$. 
There is of course some variation and there are areas and directions where the thickness is much higher or lower.}
\label{img:fnc}
\end{figure}

A simulation of~\eqref{eq:fnc} is presented in Figure~\ref{img:fnc}. 
We can see the snake-like patterns which are characteristic of the Cahn-Hilliard model.
It appears from the numerical simulations 
that the thickness of the structures is somewhat proportional to $\eps$,
as it is shown in Figure~\ref{img:fnc}.
Here we propose to measure such thickness as the distance between consecutive zeros of 
the solution considered on an arbitrary 
line through the domain, as exhibited in Figure~\ref{im:slopedline}. 

\begin{figure}[hbtp]
\centering
\includegraphics[width=\textwidth]{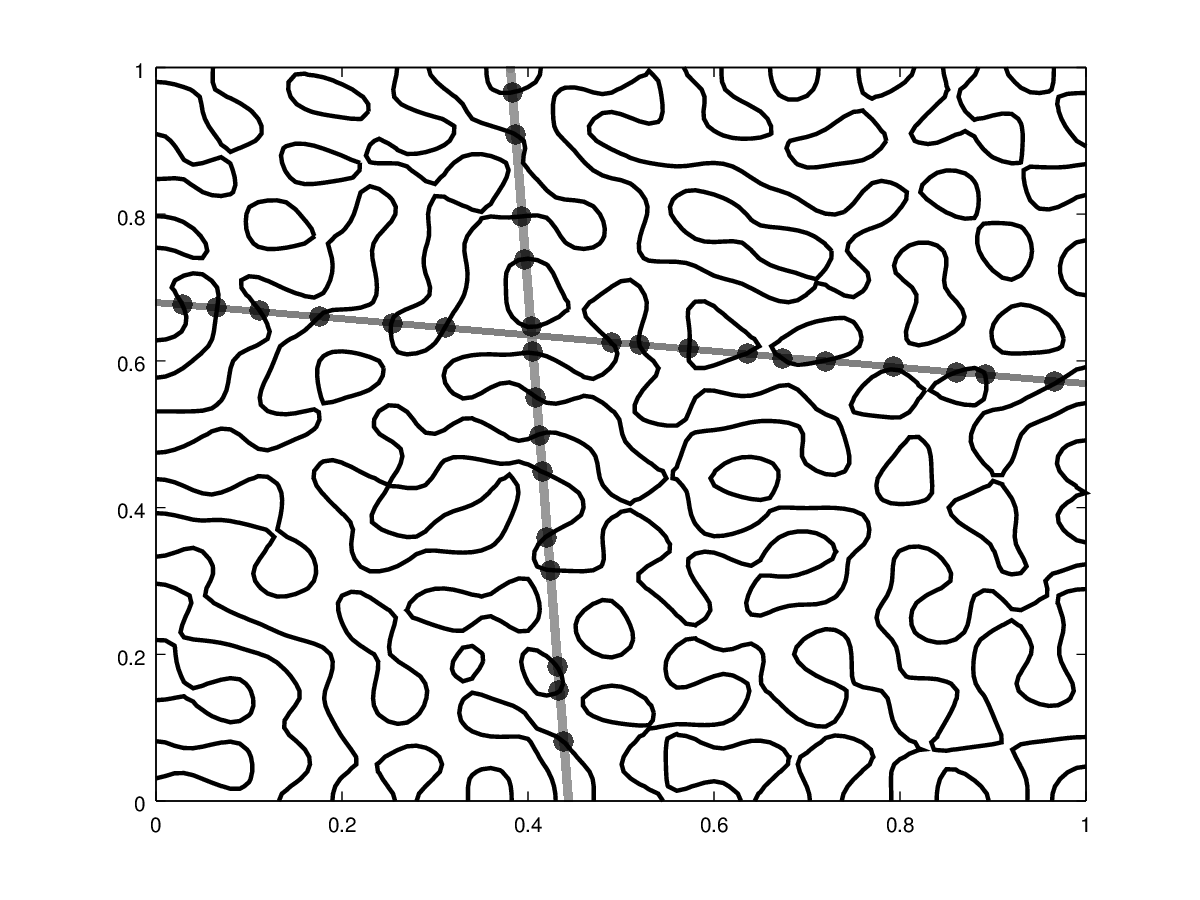}
\caption{Pattern size computation of the function $f$ 
by inserting arbitrary lines through the spatial domain $[0,1]^2$, where we represent the zero level sets for our model in the Fourier domain $R_{0.01}$. 
The zeros are order $\eps$ apart, on average.}
\label{im:slopedline}
\end{figure}

This definition of a pattern size is intuitively the right one, 
as the dominating Fourier space $R_\eps$ consists of frequencies given by the wave-vector $(k,l)\in R_\eps$ 
are all of the order of magnitude of the   wave-number $|(k,l)|$ which is $ \eps^{-1}$,
and hence the average distance is conjectured to be of the same order as $C'\eps$.
Although this line of reasoning sounds very plausible, in the wave-vector the individual frequencies can vary vastly. 
Moreover we will also see later in the main results, that when counting zeros on arbitrary straight lines, only the scaling of the 
upper bound for the wave number seems to be essential. 
Let us remark that for stationary isotropic Gaussian fields there are quite a few publications on the 
average density of zeros along lines. They already date back to the work of 
Rice \cite{rice1948statistical} or Longuet-Higgins \cite{longuet1957statistical}, \cite{longuet1957statistical2}.
See also Dennis \cite{dennis2007nodal}. 
Their approach yields qualitatively similar results of a pattern size of order $\eps$ 
for  models given by cosine series   similar to~\eqref{eq:fnc},
 but none of those methods seem to apply straightforwardly to~\eqref{eq:fnc},
 in that they usually assume isotropy and stationarity. 

A rigorous quantification of the geometry and topology of Cahn-Hilliard-Cook patterns is still an open problem, though. 
The first attempt in the setting of spinodal decomposition  by Maier-Paape and Wanner in~\cite{maier1998spinodal} yielded a partial result 
by bounding the radii of balls that would fit into the nodal domain. That strategy worked completely only in the setting of~\cite{Bl-MP:03},
where the ring $R_\eps$ was still growing, but its size was smaller by a small power in $\eps$. 

As already mentioned in the previous subsection, recent numerical  work on the number of components in the pattern has been done, 
using rigorous methods from computational algebraic topology to compute the Betti numbers, 
by Gameiro, Mischaikow and Wanner in~\cite{gammiswan05} 
or similarly by Day, Kalies, Mischaikow and Wanner in~\cite{daykalmiswan07, daykalwan09}. 
See also Guo and Hwang \cite{guohwa07} and Sander and Tatum \cite{santan12} 
for additional results on the pattern in Cahn-Hilliard equation. 
Let us recall that Betti numbers for the two-dimensional nodal domain 
count the number of connected components and the number of holes.

For general stationary or isotropic Gaussian fields there is also some work 
on Minkowski functionals measuring the area and the number of connected components 
of the nodal domain as well as the length  of the set of level $0$. See \cite{AdlerLNM}, for a collection of results.
But even the precise constant in the asymptotic behaviour of the number of connected components
seems to be still an open problem. See \cite{SodinLN} for further references.
None of these results seems to fit to \eqref{eq:fnc} anyway, not just because
of the assumptions on isotropy and stationarity, but also because we are summing over index sets that are not only growing, but also moving,
and thus the asymptotic limit is not clear.

\section{Main result}

The question we aim to rigorously answer is:
\begin{question}\label{que:main}
	What is the characteristic thickness of the pattern (i.e. snake-like structures) 
	in our model $f(x,y)$ defined in~\eqref{eq:fnc}, on the unit square $(x,y)\in[0,1]^2$?
\end{question}

In Figure~\ref{img:fnc} we can see that the average thickness appears to be $2\pi\eps$, which we will prove to be the case.

To address Question~\ref{que:main}, we take the following approach: 
we draw a straight line across the unit square and we count the (average) 
number of zeros of $f$ on that segment, see Figure~\ref{im:slopedline}.
Then we divide the length of the segment by the number of zeros, obtaining the average distance between zeros, that is the average pattern size.
So the problem reduces to counting the number of zeros of $f$, a random function, and a way of doing this is provided (in a more general form than reported here) by Edelman and Kostlan in~\cite{edelman1995many}:

\begin{theorem}[Number of real zeros of a random function]\label{thm:edelman}
	Let $v(x)=(f_0(x),\ldots,f_n(x))^T$ be any collection of differentiable functions 
	and $c_0,\ldots,c_n$ be independent and identically distributed Gaussians centred in 0.  
	Given the function
	\[
	h(x) = \sum_{k=0}^n c_k\cdot f_k(x),
	\]
	the density of real zeros of $h$ on an interval $I$ is
	\[
	\delta(x) = \frac{1}{\pi}\left\| \frac{\d}{\d x}w(x)\right\|_{\R^n},
	\qquad\textrm{where}\qquad 
	w(x)=\frac{v(x)}{\|v(x)\|_{\R^n}}.
	\]
	The expected number of real zeros of $h$ on $I$ is then
	\begin{equation*}
		\int_I \delta(x)\d x.
	\end{equation*}
\end{theorem}

Theorem~\ref{thm:edelman} holds for functions $h$ on the real line, 
so we need to translate our two-dimensional problem 
$f$ to an equivalent one formulated on a line. 
In the following first we consider $f$ constrained on an arbitrary \emph{horizontal} 
line at height $t$, which is 
\[
	L_t=\{(x,t): x\in[0,1]\}\quad\text{for}
	\quad t\in[0,1]\;.
\] 
We obtain a function of one variable and can apply Theorem~\ref{thm:edelman}. 
The generalization to arbitrary (non-horizontal) lines as the ones depicted in Figure~\ref{im:slopedline} 
is a straightforward generalization and is discussed in Section~\ref{sec:sloped}
but the underlying idea does not change. 

We need to introduce, in the spirit of Theorem~\ref{thm:edelman}, the following notation:
\begin{gather*}
	w_t(x) = \left(\frac{\cos(k\pi x)\cos(l\pi t)}{\sqrt{\sum_{m, n\in R_{\eps}} \cos^2 (m \pi x) \cos^2 (n \pi t)}}\right)_{(k,l)\in R_\eps}\\
	W_t(x) =  \left\| \left( \frac{\d}{{\d x}} w_t (x) \right)_{(k,l)\in R_\eps} \right\|^2 = \frac{S_3}{S_1}-\left(\frac{S_2}{S_1}\right)^2,
\end{gather*}
where we have
\begin{align*}
	S_1 & = \sum_{m, n\in R_{\eps}} \cos^2 (m \pi x) \cos^2 (n \pi t)\\
	S_2 & = \sum_{m, n\in R_{\eps}} m \pi \cos (m \pi x) \sin (m \pi x) \cos^2 (n \pi t)\\
	S_3 & = \sum_{m,n\in R_{\eps}}m^2\pi^2\sin^2(m\pi x)\cos^2(n\pi t).
\end{align*}

\begin{theorem}\label{thm:main}
	For any $\gamma\in(0,1)$ and any horizontal 
	line $L_t$ for $x,t\in (0,1)$ the function $W_t(x)$ defined on $L_t$
	behaves asymptotically as $(2\eps)^{-2}$ for $\eps \to 0$.
\end{theorem}

This means in particular that the average number of zeros 
is $(2\pi \eps)^{-1}$, so the mean pattern size 
(i.e. the average distance from any zero on any line to the next zero) is $2\pi \eps$.  

\begin{remark}	
	It is a remarkable fact that the result of Theorem \ref{thm:main} is independent of $\gamma$,
	because the number of Fourier modes involved is much smaller for $\gamma \approx 1$ than for $\gamma \approx 0$. 
	As we can see in Figure~\ref{im:2gammas}, while the average asymptotic pattern size along lines remains the same,
	the domain with higher $\gamma$ looks more organized. The pattern seems
	to be ``more regular'' in some sense that our method cannot detect.
\end{remark}

\begin{figure}[htbp!]
	\vspace*{-3mm}
	\centering
	\includegraphics[width=.78\textwidth]{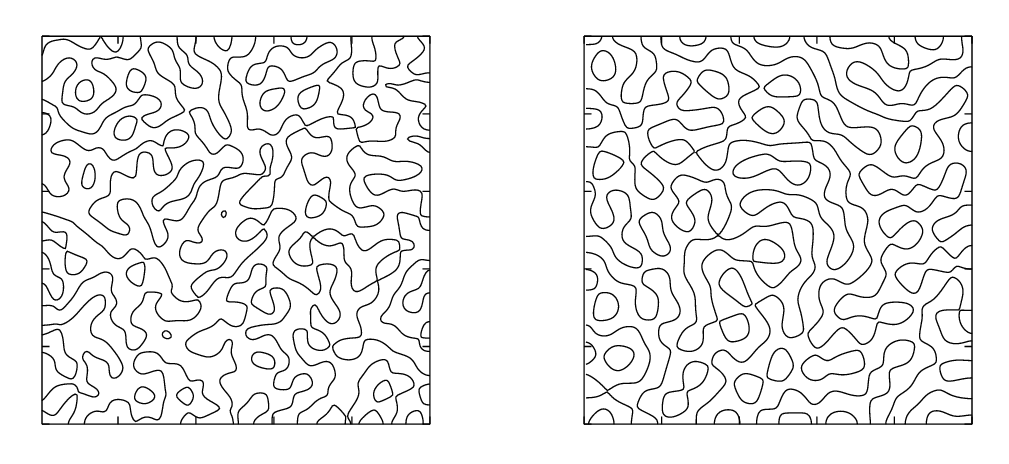}
	\vspace*{-2mm}
	\caption{Patterns generated in the Fourier domains $R_{0.01}^{0.1}$ (left) and $R_{0.01}^{0.9}$ (right)}
	\label{im:2gammas}
\end{figure}

For the proof of Theorem~\ref{thm:main} we have to show that 
\begin{align*}
	\lim\limits_{\eps\to 0} \eps\frac{S_2}{S_1}=0 & & \lim\limits_{\eps\to 0} \eps^2\frac{S_3}{S_1}=\frac{1}{4},
\end{align*}
which we do in Lemma~\ref{lem:maincorollary}. The main idea behind the proof is that the sums $S_i$, $i=1,2,3$ can be calculated in the limit via ergodic-type theorems.

\begin{remark}
	From here onwards we fix $x,t\in (0,1)$, as the functions behave a little differently on the borders. If we have $t=0,1$, we lose all the terms in $t$, that become identically 1, but we obtain anyway the same result.
	The cases $x=0,1$ are not interesting, 
	as the $\sin$ is always zero and thus $W_t(0)=W_t(1)=0$. 
	But since we are looking for a density in terms of $x$, considering it in a single point does not provide any useful information.
\end{remark}

We will prove the main result in Section \ref{sec:proof}. 
After that in Section \ref{sec:General} we state the generalization 
of the main result to general domains in Fourier space different from the quarter-ring 
and give a few explicit examples that the main result is still true, although the patterns appearing 
might look quite different.

As our result is a purely asymptotic one, we 
give a few examples in Section~\ref{sec:General}, where the number of zeros agrees very well with the asymptotic prediction even for moderate $\eps$. 
 In Section~\ref{sec:numerical},  we calculate numerically the functions $\delta(x)$ or $W_t(x)$ 
for different values of $\eps$.  There (cf. also Figure~\ref{im:deltas}) 
we see already for moderate $\eps$  a fast and uniform convergence apart from any small layer at the boundary.

Finally in Section~\ref{sec:sloped} we briefly discuss the case of sloped lines, 
which is a straightforward generalization, up to some additional technicalities. 
In the Appendix we show that we can establish convergence in the rational case not treated in the proof, where we cannot use Birkhoff's ergodic theorem.

Let us comment on further generalizations that are within the scope of this approach.
First we can easily incorporate the case where the Gaussians are not 
all identically distributed. For instance when the noise 
is diagonal w.r.t. the cosine modes,   
we just obtain additional weights in the main result, coming from the eigenvalues of the covariance operator of the noise. 

For different domains in Fourier space we always get the same order of zeros on any line through the domain,
but as we can see in Figure~\ref{im:4cases} the pattern might look quite different. 
Further characterizations using Betti numbers as in~\cite{gammiswan05},
the Minkowski-functionals as in \cite{AdlerLNM}
or bounds on the radii of balls in the nodal domains as in~\cite{maier1998spinodal}
might help to understand why the pattern in spinodal decomposition is so special, as it 
appears to be quite regular, but not too symmetric.

Moreover, we believe that the result on the square $[0,1]^2$ is not that special,
and with the same method we should be able to treat three or higher dimensional problems.
We have to stress that there are technical problems extending the result to general domains,
especially when the boundary has a complicated geometry.
In the proof for arbitrary domains we only need to replace the cosines by the
corresponding eigenfunctions. 
We should then use the ergodic theorem to recover the convergence of the sums, but there is no guarantee that we can still do that,
the main technical problem being that we cannot easily isolate the wavenumber as an argument of the function.   
Moreover, even if we could, the reduction of the two-dimensional Birkhoff ergodic theorem 
to the one dimensional case is not always possible: 
we would need to prove a true two-dimensional ergodic theorem,
but such results are available in the literature.

\section{Proof of the main result}
\label{sec:proof}

The main result is on one hand based on an application of the theorem by Edelman and Kostlan.
On the other hand, for the asymptotic behaviour we need the convergence 
of series over growing domains and the value of the limit. 
This is established by a two-dimensional weighted Birkhoff ergodic theorem.
There is a vast literature on ergodic theorems, and the theory is far developed,
so we do not attempt to give an overview.
We just refer to \cite{GoNe:09,GoNe:15} for ergodic theorems on abstract groups or subgroups of lattices. 

In the following we first state the one dimensional ergodic theorem,
then give a direct prove for a weighted ergodic theorem on arbitrary domains based 
on the analogous result on squares (see for example~\cite{hanson1969mean}).  
We also give an elementary proof that the usual ergodic theorem implies 
(under some conditions on the weights) a weighted version.
Let us remark that the results we need are not in the usual setting of weighted ergodic theorems, as for example  in the theory of ``good weights'' (see for example~\cite{baxter1983weighted, bellow1985weighted}), because we allow the weights to grow and furthermore we change the normalising constant in front of the sum.

\subsection{Ergodic theorem} 
 
Given a $\sigma$-algebra, a transformation $T$ is said to be \emph{uniquely ergodic} if it has a unique ergodic measure. The map $z\mapsto z+\alpha$ on the unit circle is uniquely ergodic if and only if $\alpha$ is irrational. In this case the unique ergodic measure is the Lebesgue measure. 

\begin{theorem}[Birkhoff ergodic theorem, see~\cite{corfomsin}]\label{thm:birkhoff}
	Let $(X,\mu)$ be a probability space. If $T$ is $\mu$-invariant and ergodic and $g$ is integrable, then for a.e. $z\in X$ 
	\begin{equation}\label{eq:thmbirkhoff}
		\lim_{N\to \infty} \frac{1}{N}\sum_{k=1}^N g\left(T^k \left(z\right)\right) = \int_X g(\zeta)\d \mu(\zeta).
	\end{equation}
	Moreover if $T$ is continuous and uniquely ergodic with measure $\mu$ and if $g$ is continuous, then~\eqref{eq:thmbirkhoff} holds for all $z\in X$ (instead of a.e.).
\end{theorem}

\begin{remark}
	Having the result to hold for all initial conditions is of paramount importance, 
	as we require it to hold for some specific initial values.
	We will then look for uniquely ergodic transformations.
\end{remark}

\begin{remark}
	If $X$ has dimension $2$, Theorem~\ref{thm:birkhoff} can be stated in the following form. 
	See for example Theorem 1.9 of \cite{GoNe:09} for an abstract version of this theorem on groups.
	\[
		\lim_{N\to\infty}\frac{1}{N^2}\sum_{k,l \in \{1,\ldots, N\}^2} (g\circ T^{(k,l)})(y,v) = \iint_X g \d \mu
	\]
	Nevertheless, we do not need an abstract proof of a two-dimensional ergodic theorem, 
	as all our functions factor, so we can always reduce the ergodic theorem on the square directly to 
	the one-dimensional case.
\end{remark}

\subsection{The weighted averaging condition}

First we draft a necessary requirement for averaged weighted sums fulfilling 
an ergodic-type property on a rectangle-shaped summation domain. 
Then we show that this can be used to obtain summation on more general domains. 
In our case this is a quarter-ring-shaped subset in $\N^2$.

After that we show that the sums we need to calculate all fulfil this requirement.

\begin{requirement}[Weighted averaging condition]
	Let $([0,1]^d, \lambda)$ be the probability space with the Lebesgue-measure $\lambda$. 
	We say that $(f, (a_m))$ with $f:[0,1]^d\to\R$ continuous and 
	extended by periodicity to $\R^d$ and $a_m\in\R$ fulfils the weighted averaging condition, 
	if for every $x^0\in [0,1]^d$, every $\alpha\in\N^d$ and 
	\[
		Q_L = \bigotimes_{i=1}^d [1,\ldots, \alpha_i L] \cap \N^d,
	\]
	the following assumption holds:
	\[
		\frac1{\sum_{m\in Q_L}a_m}\sum_{m\in Q_L} a_m\cdot f(m_1 x^0_1,m_2 x_2^0,\ldots, m_d x_d^0)
		\xrightarrow{L\to\infty} \int_{[0,1]^d} f(x)\d x\;.
	\]
\end{requirement}

For any open set $M\subset\R_+^d$ 
we define 
\[
	M_L = (L\cdot M)\cap \N^d,
\] 
the projection on the positive integers of its scaled version. 
We define 
\[
|M_L|_a = \sum_{m\in M_L} a_m
\]
and denote by 
$|M_L| $ the cardinality of $M_L $.

\begin{requirement}[Generation of measures]
	We require that the weights $a = (a_m)$ generate a measure $\lambda_a$ on $\R^d_+$ 
	which is equivalent to the Lebesgue measure $\lambda$,
	i.e. there exists an $\alpha > 0$ such that
	for each set with open interior $M\subset \R_+^d$, 
	\[
		L^{-\alpha} |M_L|_a \xrightarrow{L\to\infty} \lambda_a(M).
	\]
\end{requirement}
Let us remark that in general, we could get the result for a weaker assumption 
on the measure $\lambda_a$ than being equivalent to Lebesgue measure.
But as in all our examples this is the case, we assume this for simplicity of presentation.

The trivial example for the generation of measures are the constant 
weights $a_{k,l}=1$ that generate the Lebesgue measure
with $\alpha=d$. 

\begin{example}
	 An example that is used frequently later is $a_{k,l}=k^2$ for dimension $d=2$. 
	 Then by Riemann sum approximation
	\[
		L^{-4}|M_L|_a 
		= L^{-4} \sum_{(k,l)\in L\cdot M} k^2 
		=   \sum_{(k,l)\in M \cap\frac1{L}\N^2} k^2 L^{-2}
		\xrightarrow{L\to\infty} \int_M \xi^2 \d(\xi,\eta)\;.
	\]
	Thus the measure $\lambda_{(k^2,1)}$ generated by the weight has a Lebesgue-density $ (\xi,\eta) \mapsto \xi^2$.
	As the density is up to the Lebesgue null set $\xi=0$  everywhere strictly positive, 
	the measures $\lambda_{(k^2,1)}$ and $\lambda$ are equivalent. 
\end{example}

\begin{lemma}\label{lem:domainexpansion}
	Let $(f, (a_m))$ fulfil the weighted averaging condition
	with the weights $a$ generating a measure. 
	Then for any open measurable set $ S\subset\R_+^d$ 
	\[
		\frac1{|S_L|_a}\sum_{m\in S_L} a_m\cdot f(m_1 x^0_1,m_2 x_2^0,\ldots, m_d x_d^0)
		\xrightarrow{L\to\infty} \int_{[0,1]^d}f(x)\d x\;.
	\]
\end{lemma}
\begin{proof}
	We prove this in two steps. We call $S_L$ the summation domain. 
	First we show that the weighted averaging condition also holds 
	for summation domains which are rectangles not aligned at the origin 
	(i.e. we shift the $Q_L$ out of the origin). Second, we cover $S$
	by a disjoint union of such rectangles and conclude the proof.
	\begin{figure}[htbp!]
		\centering
		\begin{tikzpicture}
			\draw (0,0) rectangle +(5,4);
			\draw[dashed] (0.1, 0.1) rectangle + (4.8,2);
			\draw[dotted] (0.2, 0.2) rectangle + (1, 3.7);
			\draw[dashdotted] (0.3, 0.3) rectangle + (0.8, 1.7);
			\draw (1.3, 2.2) rectangle + (3.6, 1.7);
			\draw (2.4, 1) node {$Q_L^1$};
			\draw (0.7, 2.8) node {$Q_L^2$};
			\draw (5.5, 2) node {$Q_L^3$};
			\draw (3, 3) node {$R_L$};
		\end{tikzpicture}
		\caption{Construction of $R_L$ from rectangles with one vertex in the origin: $R_L = Q_L^3\setminus (Q_L^1\cup Q_L^2)$}
		\label{im:construction}
	\end{figure}
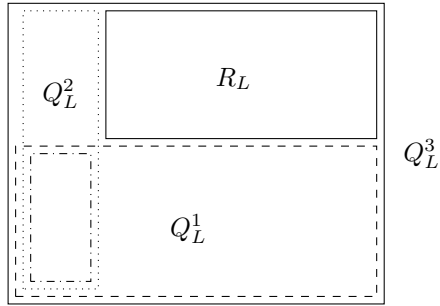
	Consider the scaled rectangle
	$R_L = \otimes_{i=1}^d \{\beta_i L,\ldots,\gamma_i L\}\subset \N^d$.
	Then $R_L$ can be constructed by making unions and subtractions of origin-aligned rectangles 
	of the form $Q_L$ as depicted in Figure \ref{im:construction}. 
	Thus
	\[
		\frac1{|R_L|_a}\sum_{m\in R_L} a_m\cdot f(m_1 x^0_1,m_2 x_2^0,\ldots, m_d x_d^0)
		\xrightarrow{L\to\infty} \int_{[0,1]^d} f(x)\d x,
	\]
	follows immediately.
	
	For every approximation threshold $\delta > 0$ there are finite sets of squares
	$\{P_i\}_i$ and $\{O_i\}_i$ fulfilling
	\[
		{M}_\delta = \bigcup_i {P}_i \subset {S} \subset \bigcup_i {O}_i 
		= {N}_\delta.
	\]
	with $\lambda({N}_\delta \setminus S) \leq \lambda({N}_\delta \setminus {M}_\delta) < \delta$. 
	Write $(M_{\delta, L} = {M}_\delta \cdot L)\cap \N^d$, 
	$P_{i, L} = ({P}_i\cdot L)\cap \N^d$ and $N_{\delta, L} = ({N}_\delta \cdot L)\cap \N^d$ 
	analogously to ${S}$ and $S_L$. We write $f(m\odot x^0) = f(m_1 x^0_1,m_2 x_2^0,\ldots, m_d x_d^0)$ for brevity.
	Then we can derive the following error estimate
	\begin{equation*}
		\begin{split}
			\operatorname{Err} &=\left| \frac1{|N_{\delta, L}|_a}\sum_{m\in N_{\delta, L}}a_m f\left(m\odot x^0\right) - \frac1{|S_L|_a}\sum_{m\in S_L}a_mf\left(m\odot x^0\right) \right| \\
			&=\frac{1}{|N_{\delta, L}|_a\cdot|S_L|_a}\left| |S_L|_a \sum_{m\in N_{\delta, L}} a_mf\left(m\odot x^0\right) - | N_{\delta, L}|_a\sum_{m\in S_L}a_mf\left(m\odot x^0\right)\right|\\
			&=\frac{1}{|N_{\delta, L}|_a\cdot|S_L|_a}\left| \left(|N_{\delta, L}|_a-|S_L|_a\right)\sum_{m\in  S_L}a_mf\left(m\odot x^0\right) +|S_L|_a\sum_{m\in N_{\delta, L}\setminus S_L}a_m f\left(m\odot x^0\right)\right|\\
			&\leq\frac{1}{|N_{\delta, L}|_a\cdot|S_L|_a}\left| \left(|N_{\delta, L}|_a-|S_L|_a\right)\|f\|_\infty \! \sum_{m\in  S_L}a_m\right| +\frac{1}{|N_{\delta, L}|_a\cdot|S_L|_a} \left||S_L|_a\cdot \|f\|_\infty \! \!\!\sum_{m\in N_{\delta, L}\setminus S_L}a_m \right|\\
			 &= 2\|f\|_\infty \cdot \frac{| |N_{\delta, L}|_a - |S_L|_a|}{|N_{\delta, L}|_a} 
			 = 2\|f\|_\infty \cdot \frac{| N_{\delta, L}\setminus S_L|_a}{|N_{\delta, L}|_a} \xrightarrow{L\to\infty} 2\|f\|_\infty \cdot \frac{\lambda_a(N_\delta\setminus S)}{\lambda_a(N_\delta)} 
			 \xrightarrow{\delta\to0}0, 
		\end{split}
	\end{equation*}
	where in the last step we use that the weight $a$ generates a measure $\lambda_a$ 
	equivalent to the Lebesgue measure. 
	Thus $\lambda_a(N_\delta) \to \lambda_a(S) >0$ and   $\lambda_a(N_\delta\setminus S) \to\lambda_a(\varnothing)=0$.
	
	Hence the approximation error has for all $\delta>0$ a limit for $L\to\infty$ 
	which is arbitrarily small for $\delta\to0$, and it remains to check the limit on 
	the cubical approximation $N_{\delta}$ instead of $S$, but this is straightforward:
	\[
		 \frac{1}{|N_{\delta, L} |}  \sum_{m\in M_{\delta, L}} f (m\odot x^0) =
		 \frac1{|N_{\delta,L}|}\sum_{\substack{P_{i,L} \\
		 \cup_iP_{i,L} =N_{\delta, L}}} |P_{i,L}|\cdot \frac1{|P_{i,L}|}\sum_{m\in P_{i,L}}f(m\odot x^0)
		 \xrightarrow{L\to\infty} \int_{[0,1]^d} f(x) \d x,
	\]
	according to the first step of this proof.
\end{proof}

\subsection{Weighted averages}

Here we present an elementary result showing that if the averages converge
then the weighted averages do, too.

\begin{lemma}\label{lem:weightedzerosequence}
	Given two sequences $\{b_k\}_k$ and $\{f_k\}_k$, such that $b_k \geq 0$ for all $k$,  if 
	\[
		F_N = \frac1{N}\sum_{k=1}^N f_k \xrightarrow{N\to\infty} 0,
	\]   
	and also $\frac{a_N \cdot N}{\sum_{k=1}^N a_k}\xrightarrow{N\to\infty} a$, for some $a\in \R$, with $a_k = \sum_{n=1}^k b_n$, and $\sum_{k=1}^Na_k\xrightarrow{N\to\infty} \infty$,
	then
	\[
		\frac{\sum_{k=1}^N a_k\cdot  f_k}{\sum_{k=1}^N a_k} \xrightarrow{N\to\infty} 0.
	\]
\end{lemma}
\begin{proof}
	Consider first
	\begin{align*}
		\sum_{n=1}^N n\cdot b_n = \sum_{n=1}^N b_n \cdot N - \sum_{n=1}^{N-1}1 \cdot \sum_{k=1}^nb_k = N \cdot a_N - \sum_{n=1}^{N-1} a_n,
	\end{align*}
	which means that
	\[
		\frac{\sum_{n=1}^N n\cdot b_n}{\sum_{n=1}^N a_n}\xrightarrow{N\to\infty} a-1.
	\]
	Now
	\begin{equation*}
		\begin{split}
			\frac{\sum_{k=1}^N a_k f_k}{\sum_{k=1}^N a_k} &= \frac{\sum_{k=1}^N \sum_{n=1}^k b_n \cdot f_k}{\sum_{k=1}^N a_k}\\
			&= \frac{\sum_{n=1}^N b_n \cdot \sum_{k=n}^N f_k}{\sum_{k=1}^N a_k}\\
			&= \frac{N}{\sum_{k=1}^N a_k} \sum_{n=1}^N \left[b_n \cdot F_N - b_n \cdot \frac{n}{N}\cdot F_{n}\right] \\
			&= F_N\cdot \frac{N \cdot a_N}{\sum_{k=1}^N a_k} - \sum_{n=1}^N F_{n} \cdot \frac{n\cdot b_n}{ \sum_{k=1}^N a_k}. 
		\end{split}
	\end{equation*}
	Note that $\sum_{n=1}^N F_{n} \cdot \frac{n\cdot b_n}{ \sum_{k=1}^N a_k}$ is bounded by a constant $C$ and there is $N_1$ such that $|F_N\cdot \frac{N \cdot a_N}{\sum_{k=1}^N a_k}| < \frac{\eps}{3}$ for $n > N_1$. Also, there exists $N_2$, such that for $n> N_2$ we have 
\[
\left|\sum_{n=1}^{N_2} F_{n} \cdot \frac{n\cdot b_n}{ \sum_{k=1}^N a_k} \right| \leq C\cdot\frac{\sum_{k=1}^{N_2} a_k}{\sum_{k=1}^N a_k}\leq \frac{\eps}{3}.
\]	
Lastly we can choose $N_3 $ such that for $N\geq N_3$
	\[
		\left|\sum_{n=1}^{N_2} F_{n} \cdot \frac{n\cdot b_n}{ \sum_{k=1}^N a_k}\right| < \frac{\eps}{3}.
	\]
	Then
	\begin{align*}
		\left| \frac{\sum_{k=1}^N a_k f_k}{\sum_{k=1}^N a_k}\right| &\leq \left| F_N\cdot \frac{N \cdot a_N}{\sum_{k=1}^N a_k}\right| + \left|\sum_{n=1}^{N_2} F_{n} \cdot \frac{n\cdot b_n}{ \sum_{k=1}^N a_k}\right| + \left|\sum_{n=N_2+1}^{N} F_{n} \cdot \frac{n\cdot b_n}{ \sum_{k=1}^N a_k}\right|\\
		& < \eps
	\end{align*}
	for $n > \max(N_1,N_2,N_3)$.
\end{proof}
The following corollary follows immediately by setting $f = g-C$.
\begin{corollary}\label{cor:weightedaverage}
	Let the coefficients $a_k$ be as in Lemma \ref{lem:weightedzerosequence}. If 
	\[
		\frac1{N}\sum_{k=1}^N g_k \xrightarrow{N\to\infty} C,
	\]
	then
	\[
		\frac{\sum_{k=1}^N a_k\cdot  g_k}{\sum_{k=1}^N a_k} \xrightarrow{N\to\infty} C.
	\]
\end{corollary}

\subsection{Asymptotic behaviour of the \texorpdfstring{$S_i$}{Si}}

Now we turn to the main technical tool to prove the main result.

\begin{lemma}\label{lem:maincorollary} 
	For $x,t\in (0,1)$,
	\begin{gather}
		\label{eq:S1}
		\lim_{\eps\to 0} \frac{1}{|R_\eps|}\cdot S_1 = \frac{1}{4}\\
		\label{eq:S3}
		\lim_{\eps\to 0} \eps^2\cdot\frac{S_3}{S_1}= \frac{1}{4}\\
		\label{eq:S2}
		\lim_{\eps\to 0}\eps\cdot\frac{S_2}{S_1} = 0.
	\end{gather}
\end{lemma}

\begin{remark}
	In the following proof we will show that the weighted averaging condition holds for $x,t\notin\Q$ by using Birkhoff's ergodic theorem~\ref{thm:birkhoff}. This leaves out the case when either $x$ or $t$ are in~$\Q$. We have two ways of addressing this issue. We can compute explicitly the terms and obtain analytically the convergence required for the weighted average condition; these are straightforward but tedious computations, a sample of which is provided in the appendix, in Subsection~\ref{sec:rational}. 
	
	At the same time we can observe that it is enough to have the limits almost everywhere in $x,t$, as we are integrating $W_t(x)$ and the Lebesgue integral ignores sets of null measure.
\end{remark}

\begin{proof}[Proof of Lemma~\ref{lem:maincorollary}]
	As discussed above, we will consider only the case $x,t\notin\Q$.
	
	We first prove~\eqref{eq:S1}. We can use Birkhoff's ergodic theorem. 
	Define $T_x(z) = z+x$: this is a measure-preserving and uniquely ergodic transformation (since $x\not\in \Q$). 
	Then
	\[
		\frac{1}{N}\sum_{k=0}^{N-1}\cos^2(\pi k x) = \frac{1}{N}\sum_{k=0}^{N-1}\cos^2\left(\pi T_x^k\left(0\right)\right) \xrightarrow{N\to\infty} \int_0^1\cos^2(\pi x)\d x = \frac{1}{2}.
	\]
	All the coefficients $a_m$ are 1 in this case and the function is multiplicative.
	Then the result follows immediately from Lemma \ref{lem:domainexpansion} and the fact that 
	\[
		\int_{[0,1]^2}\cos^2(\pi x_1)\cos^2(\pi x_2)\d (x_1,x_2) = \frac1{4}.
	\]
	We prove now~\eqref{eq:S3}. For $x,t\not\in\Q$ we know from~\eqref{eq:S1} 
	that the denominator $S_1 \sim \frac{|R_\eps|}{4}$ as $\eps\to 0$ for $x\neq 0,1$. 
	Hence the term in question has the same asymptotic behaviour as
	\begin{equation}\label{e:asy}
		\eps^2\frac{S_3}{S_1} \sim 4\pi^2 \frac{\eps^2}{|R_\eps|}\cdot \sum_{k,l\in R_\eps} k^2\sin^2(k\pi x)\cos^2(l\pi t).
	\end{equation}
	We define $a = (a_{k,l})_{k,l}$ 
	with $a_{k,l} = k^2$, where we already saw in our example that his generates a measure 
	with Lebesgue density 
	such that 
	\begin{equation}\label{eq:quotient}
		\lim_{\eps\to 0}\eps^2\cdot \frac{|R_\eps|_a}{|R_\eps|}
		=  \lim_{\eps\to 0}\eps^2\cdot \frac{1}{|R_\eps|}\sum_{k,l\in R_\eps} k^2 
		=  \frac{\lambda_a(R)}{\lambda(R)},
	\end{equation}
	where the rescaled domain is
	\[
		R  = \{(\eta,\xi)\in\R^2|~ \alpham < \sqrt{\xi^2+\eta^2} <  \alphap\},
	\]
	which gives 
	\[
		R_\eps = \eps^{-1}R\cap \N^2. 
	\]
	As $ \d \lambda_a =  \eta^2 \d(\eta,\xi)$ we obtain by elementary calculations 
	using polar coordinates
	\begin{align*}
		\frac{\lambda_a(R)}{\lambda(R)} &=  \frac4{ \pi (\alphap^2-\alpham^2)} \int_R \eta^2 \d(\eta,\xi)\\
		&=  \frac4{ \pi (\alphap^2-\alpham^2)} \int_0^{\pi/2} \int_{\alpham}^{\alphap} r^3 \cos(\varphi)^2\d r\d\varphi
		=  \frac1{\alphap^2-\alpham^2}\int_{\alpham}^{\alphap} r^3  \d r\\
		&=  \frac14 \frac{\alphap^4-\alpham^4}{\alphap^2-\alpham^2} =  \frac14 (\alphap^2+\alpham^2) = \frac1{4\pi^2}\;.
	\end{align*}
	Combining \eqref{e:asy} and \eqref{eq:quotient}, it remains to show that we can apply the 
	averaged Birkhoff ergodic theorem to the sum.  We already saw that it is sufficient to check this on large rectangles 
	containing the origin.
	We define $Q_L$ as in the weighted averaging condition for $\alpha \in(0,\infty)^2$ so that
	\[
		Q_L =  [1,\alpha_1 L] \times [1,\alpha_2 L] \cap \N^2 
		= I_L^{(1)} \times I_L^{(2)} \quad\text{with} \quad I_L^{(j)} = [1,\alpha_j L] \cap \N \;.
	\]
	Because of the rectangular shape we can split the sum
	\begin{align*}
		 \frac{1}{|Q_L|_a}\cdot \sum_{(k,l)\in Q_L}  k^2\sin^2(k\pi x)\cos^2(l\pi t)
		& = \frac1{\sum_{k\in I_L^{(1)}} k^2} \cdot  \sum_{k\in I_L^{(1)}} k^2\sin^2(k\pi x)
		\cdot \frac{1}{\sum_{l\in I_L^{(2)}} 1}\sum_{l\in I_L^{(2)}} \cos^2(l\pi t)\\
		&\sim \frac1{|I_L^{(1)}|}\cdot  \sum_{k\in I_L^{(1)}} \sin^2(k\pi x)
		\cdot \frac1{|I_L^{(2)}|}\sum_{l\in I_L^{(2)}} \cos^2(l\pi t)\\
		& \xrightarrow{L\to\infty} \frac{1}{4}\;,
	\end{align*}
	by the standard one-dimensional ergodic theorem,
	where we used Corollary~\ref{cor:weightedaverage} to remove the weights. 
	We can now use Lemma~\ref{lem:domainexpansion} to extend this to the ring and obtain
	\[
		\eps^2\cdot \frac{S_3}{S_1}
		\sim 4\pi^2 \eps^2 \cdot \frac{|R_\eps|_a}{|R_\eps|} 
		\cdot \frac{1}{|R_\eps|_a}\cdot \sum_{k,l\in R_\eps}k^2\sin^2(k\pi x)\cos^2(l\pi t)
		\xrightarrow{\eps\to 0} \frac{1}{4},
	\]
	where we evaluated the value of the limit in~\eqref{eq:quotient}.
	
	Finally, we prove~\eqref{eq:S2}. As it was the case in~\eqref{eq:S3}, we can calculate the sum of the coefficients, factor them out and see that the orders of magnitude of $\eps$ cancel out. Then everything is reduced to the integral
	\[
		\int_{[0,1]^2}\cos(\pi x_1)\sin(\pi x_1)\cos^2(\pi x_2) \d(x_1,x_2) = 0.\qedhere
	\]
\end{proof}

\subsection{Proof of main result}
We have now all the ingredients to complete the proof of Theorem~\ref{thm:main}.

\begin{proof}[Proof of Theorem~\ref{thm:main}]
	From Lemma~\ref{lem:maincorollary} we know that
	\begin{equation*}
		W_t(x) = \frac{S_3}{S_1}-\left(\frac{S_2}{S_1}\right)^2\sim \frac{1}{4\eps^2}\qquad \textrm{as}\ \eps\to 0.
	\end{equation*}
	By Theorem~\ref{thm:edelman}, the number of expected zeros on the horizontal line $L_t$ of length 1 is, for a given $\eps$,
	\begin{equation*}
		N = \frac{1}{\pi}\int_0^1 \sqrt{\frac{1}{4\eps^2}}\d x = \frac{1}{2\pi\cdot\eps}.
	\end{equation*}
	This is the same as saying that the average pattern size is $\frac{1}{N}=2\pi\eps$.
\end{proof}

\section{General Fourier domains} \label{sec:General}

We can consider domains in Fourier space that are different from the ring. 
The main difference is that we get a different constant, which depends on the shape of the Fourier-domain.
Also the constant might be different along different directions 
if we lose the symmetry in $k$ and $l$ in the Fourier space, but we will discuss that later in Section \ref{sec:sloped}.
On the other hand for horizontal lines we can give an explicit constant,
in terms of the measure $\lambda_a$, and for vertical lines we can argue by symmetry.
Let us discuss a few explicit examples.

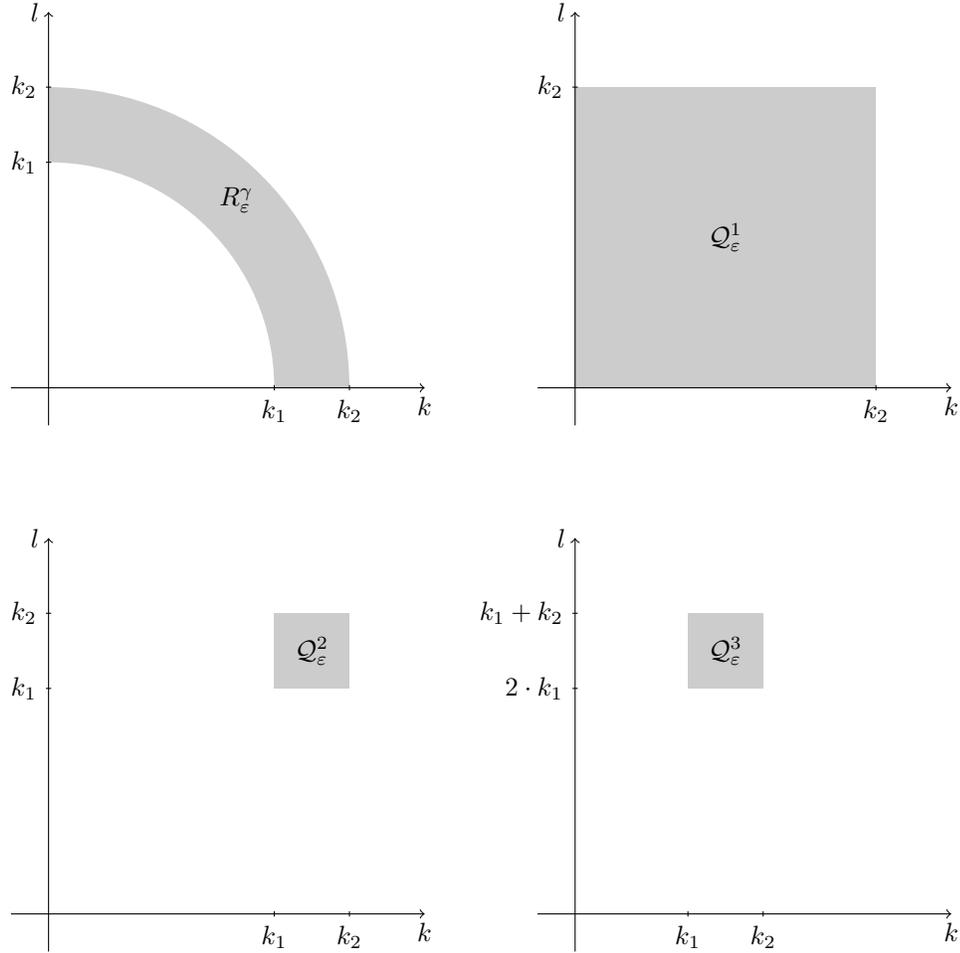
\begin{figure}[t!]
	\begin{tikzpicture}
		\fill[black!20!white] (4, 0) arc (0:90:4) -- (0, 3) arc (90:0:3) circle;
		\draw[->] (-.5, 0) -- (5, 0);
		\draw (5,0) node[anchor=north] {$k$};
		\draw[->]  (0, -.5) -- (0, 5);
		\draw (0,5) node[anchor=east] {$l$};
		\draw (3,1pt) -- (3,-1pt) node[anchor=north] {$k_1$};
		\draw (4,1pt) -- (4,-1pt) node[anchor=north] {$k_2$};
		\draw (1pt, 3) -- (-1pt, 3) node[anchor=east] {$k_1$};
		\draw (1pt, 4) -- (-1pt, 4) node[anchor=east] {$k_2$};
		\draw (2.5, 2.5) node {$R^\gamma_\eps$};
		\begin{scope}[xshift = 7cm]
			\fill[black!20!white] (0,0) rectangle (4, 4);
			\draw[->] (-.5, 0) -- (5, 0);
			\draw (5,0) node[anchor=north] {$k$};
			\draw[->]  (0, -.5) -- (0, 5);
			\draw (0,5) node[anchor=east] {$l$};
			\draw (4,1pt) -- (4,-1pt) node[anchor=north] {$k_2$};
			\draw (1pt, 4) -- (-1pt, 4) node[anchor=east] {$k_2$};
			\draw (2, 2) node {$\mathcal{Q}^1_\eps$};
		\end{scope}
		\begin{scope}[yshift = -7cm]
			\fill[black!20!white] (3,3) rectangle (4, 4);
			\draw[->] (-.5, 0) -- (5, 0);
			\draw (5,0) node[anchor=north] {$k$};
			\draw[->]  (0, -.5) -- (0, 5);
			\draw (0,5) node[anchor=east] {$l$};
			\draw (3,1pt) -- (3,-1pt) node[anchor=north] {$k_1$};
			\draw (4,1pt) -- (4,-1pt) node[anchor=north] {$k_2$};
			\draw (1pt, 3) -- (-1pt, 3) node[anchor=east] {$k_1$};
			\draw (1pt, 4) -- (-1pt, 4) node[anchor=east] {$k_2$};
			\draw (3.5, 3.5) node {$\mathcal{Q}^2_\eps$};
		\end{scope}
		\begin{scope}[xshift = 7cm, yshift = -7cm]
			\fill[black!20!white] (1.5, 3) rectangle (2.5, 4);
			\draw[->] (-.5, 0) -- (5, 0);
			\draw (5,0) node[anchor=north] {$k$};
			\draw[->]  (0, -.5) -- (0, 5);
			\draw (0,5) node[anchor=east] {$l$};
			\draw (1.5,1pt) -- (1.5,-1pt) node[anchor=north] {$k_1$};
			\draw (2.5,1pt) -- (2.5,-1pt) node[anchor=north] {$k_2$};
			\draw (1pt, 3) -- (-1pt, 3) node[anchor=east] {$2\cdot k_1$};
			\draw (1pt, 4) -- (-1pt, 4) node[anchor=east] {$k_1+k_2$};
			\draw (2, 3.5) node {$\mathcal{Q}^3_\eps$};
		\end{scope}
	\end{tikzpicture}
	\caption{The ring $R^\gamma$ and the alternative mode domains $\mathcal{Q}^i$, $i=1,2,3$.}\label{im:domains}
\end{figure}

Consider the following mode domains where we define as before for some $\gamma\in(0,1)$ 
\[k_1= \frac{\alpham}{\eps} = \sqrt{\frac{1-\sqrt{1-\gamma}}{2\pi^2\eps^2}} 
  \quad\text{and}\quad
  k_2= \frac{\alphap}{\eps} = \sqrt{\frac{1+\sqrt{1-\gamma}}{2\pi^2\eps^2}}
\]

\begin{align*}
R^\gamma_\eps&= \{(k,l)\in\N^2|~ k_1 < \sqrt{k^2+l^2} <  k_2\},\\
\mathcal{Q}^1_\eps&= \{(k,l)\in\N^2|~ 0 < k,l <  k_2\},\\
\mathcal{Q}^2_\eps&= \{(k,l)\in\N^2|~ k_1 < k,l <  k_2\},\\
\mathcal{Q}^3_\eps&= \{(k,l)\in\N^2|~ k_1< k <  k_2,\, 2k_1<l< k_1+k_2\}.
\end{align*}
Note that $R^\gamma_\eps$ is the same quarter-ring-shaped domain as before and the others are certain rectangles, all of them represented in Figure~\ref{im:domains}.
For our main result, 
we will also use the scaled  $\eps$-independent versions 
\[
R^\gamma = \{(\eta,\xi)\in\R^2|~ \alpham < \sqrt{\xi^2+\eta^2}  < \alphap\}\\
\]
\[
\mathcal{Q}^1  =(0,\alphap)^2, \quad
\mathcal{Q}^2 = (\alpham, \alphap)^2, \quad
\mathcal{Q}^3 =  (\alpham, \alphap) \times (2\alpham, \alpham +\alphap),
\]
so that
\[
R^\gamma_\eps = \eps^{-1} R^\gamma \cap \N^2
 \quad \text{and}\quad  \mathcal{Q}^i_\eps =\eps^{-1}\mathcal{Q}^i\cap \N^2\text{ for }i=1,2,3\;.
\]

We provide numerical simulations on how the patterns look like on such domains in Figure~\ref{im:4cases}.
They look quite different, although as we will see in the following Lemma~\ref{lem:following},
the main result about the asymptotic distribution of zeros is the same in all four cases, only the constants change.

\begin{lemma}\label{lem:following}
Let $D_\eps = \eps^{-1} D\cap \N^2$ be a scaled domain in Fourier space
(for example $R^\gamma_\eps$ or $\mathcal{Q}^i_\eps$ for $i=1,2,3$).
Then the asymptotic (for $\eps\to 0$) density of zeros $\delta$ is on horizontal lines through the pattern 
\begin{displaymath}
\delta(x) \sim  \frac{1}{2\pi\eps}\cdot
\sqrt{ \frac{\lambda_{(k^2,1)}(D)}{\lambda(D)}},
\end{displaymath}
while on vertical lines it is 
\begin{displaymath}
\delta(x) \sim  \frac{1}{2\pi\eps}\cdot
\sqrt{ \frac{\lambda_{(1,l^2)}(D)}{\lambda(D)}}.
\end{displaymath}
\end{lemma}
\begin{proof}
This is a proof analogous to the one of Lemma~\ref{lem:maincorollary}.
We obtain the following asymptotic equivalences
\[
\delta(x)^2  
\sim \frac{S_3}{S_1} 
\sim \frac{1}{4}\cdot \frac{|D_\eps|_{(k^2,1)}}{|D_\eps|}
\sim \frac{1}{4\eps^2}\cdot \frac{\lambda_{(k^2,1)}(D)}{\lambda(D)}.
\]
The statement on vertical lines simply follows by symmetry. 
\end{proof}

The correction factors $\lambda_a(D)/\lambda(D)$  for the weights $a_{k,l}=k^2$ 
and the pattern sizes in the following domains are 
\begin{center}
{\setlength\extrarowheight{10pt}
\begin{tabular}{r|l|l|l}
\textnormal{Domain} & \textnormal{Correction coeff.} & \textnormal{Avg. number of zeros} & \textnormal{Avg. pattern size} \\
\hline
$R^\gamma_\eps$ & $1$ & $\dfrac{1}{2\pi \eps}$ & $2\pi \eps$\\[8pt]
$\mathcal{Q}^1_\eps$ & $\dfrac{2}{3}(1+\sqrt{1-\gamma})$ & $\dfrac{\sqrt{\dfrac{2}{3}(1+\sqrt{1-\gamma})}}{2\pi\eps}$& $\dfrac{2\pi\eps}{\sqrt{\dfrac{2}{3}(1+\sqrt{1-\gamma})}}$\\
$\mathcal{Q}^2_\eps$ & $\dfrac{2}{3}(2+\sqrt{\gamma})$ & $\dfrac{\sqrt{\dfrac{2}{3}(2+\sqrt{\gamma})}}{2\pi\eps}$ & $\dfrac{2\pi\eps}{\sqrt{\dfrac{2}{3}(2+\sqrt{\gamma})}}$
\end{tabular}\vspace{4pt}
}
\end{center}
These are due to symmetry the same results for horizontal or vertical lines,
while for the non-symmetric $\mathcal{Q}^3$ we have
\begin{center}
{\setlength\extrarowheight{10pt}
\begin{tabular}{r|l|l|l}
 $\mathcal{Q}^3_\eps$ & \textnormal{Correction coeff.} & \textnormal{Avg. number of zeros} & \textnormal{Avg. pattern size} \\
\hline
horizontal & $\dfrac{2}{3}(2+\sqrt{\gamma})$ & $\dfrac{\sqrt{\dfrac{2}{3}(2+\sqrt{\gamma})}}{2\pi \eps}$& $\dfrac{2\pi \eps}{\sqrt{\dfrac{2}{3}(2+\sqrt{\gamma})}}$\\
vertical & $\dfrac{2}{3}(8-6\sqrt{1-\gamma}+4\sqrt{\gamma})$ & $\dfrac{\sqrt{\dfrac{2}{3}(8-6\sqrt{1-\gamma}+4\sqrt{\gamma})}}{2\pi\eps}$& $\dfrac{2\pi\eps}{\sqrt{\dfrac{2}{3}(8-6\sqrt{1-\gamma}+4\sqrt{\gamma})}}$.
\end{tabular}\vspace{4pt}
}
\end{center}

\begin{figure}[htbp!]
\centering
\hspace{-.065\textwidth}\scalebox{.78}{
\setlength{\unitlength}{1pt}
\begin{picture}(0,0)
\includegraphics{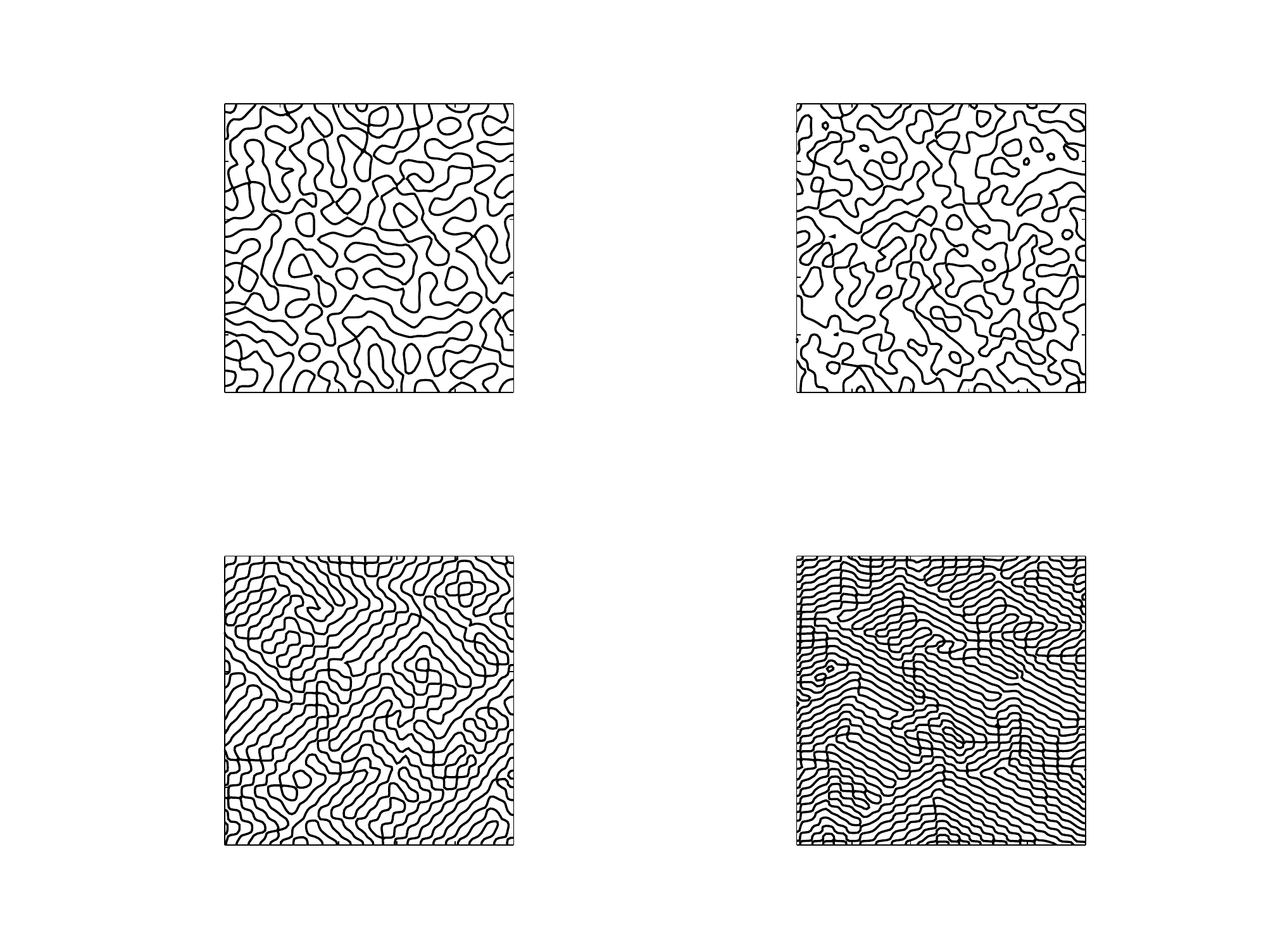}
\end{picture}%
\begin{picture}(576,433)(0,0)

\fontsize{10}{0}
\selectfont\put(235.051,250.943){\makebox(0,0)[t]{\textcolor[rgb]{0,0,0}{{$x$}}}}
\fontsize{10}{0}
\selectfont\put(98.5623,386.431){\makebox(0,0)[r]{\textcolor[rgb]{0,0,0}{{$y$}}}}

\selectfont\put(498.057,250.943){\makebox(0,0)[t]{\textcolor[rgb]{0,0,0}{{$x$}}}}
\fontsize{10}{0}
\selectfont\put(360.569,386.431){\makebox(0,0)[r]{\textcolor[rgb]{0,0,0}{{$y$}}}}

\fontsize{10}{0}
\selectfont\put(235.051,45.3292){\makebox(0,0)[t]{\textcolor[rgb]{0,0,0}{{$x$}}}}
\fontsize{10}{0}	
\selectfont\put(98.5623,178.818){\makebox(0,0)[r]{\textcolor[rgb]{0,0,0}{{$y$}}}}

\fontsize{10}{0}
\selectfont\put(498.057,45.3292){\makebox(0,0)[t]{\textcolor[rgb]{0,0,0}{{$x$}}}}
\fontsize{10}{0}
\selectfont\put(360.569,178.818){\makebox(0,0)[r]{\textcolor[rgb]{0,0,0}{{$y$}}}}
\fontsize{10}{0}

\selectfont\put(163.297,396.431){\makebox(0,0)[b]{\textcolor[rgb]{0,0,0}{{$R^{0.7}_{0.01}$}}}}
\fontsize{10}{0}
\selectfont\put(416.303,396.431){\makebox(0,0)[b]{\textcolor[rgb]{0,0,0}{{${\mathcal{Q}}^1_{0.01}$}}}}
\fontsize{10}{0}
\selectfont\put(163.297,190.0){\makebox(0,0)[b]{\textcolor[rgb]{0,0,0}{{${\mathcal{Q}}^2_{0.01}$}}}}
\fontsize{10}{0}
\selectfont\put(416.303,190.0){\makebox(0,0)[b]{\textcolor[rgb]{0,0,0}{{${\mathcal{Q}}^3_{0.01}$}}}}
\end{picture}

}
\caption{Patterns generated by the Fourier domains presented in Figure~\ref{im:domains} }
\label{im:4cases}
\end{figure}
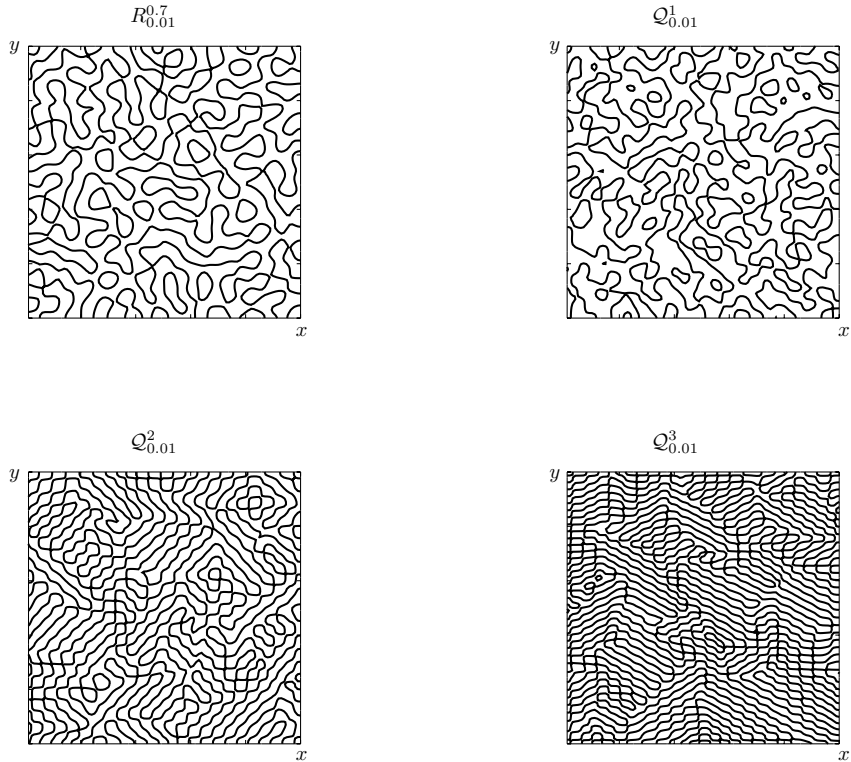

In Figure~\ref{im:4cases} the simulations are run with $\gamma=0.7$ and $10^{-2}$. 
In those cases we have the following rounded off asymptotic values for the number of zeros on a line:
\begin{center}
{\setlength\extrarowheight{4pt}
\begin{tabular}{l|l|l|l}
\textnormal{Domain} & \textnormal{Correction coeff.} & \textnormal{Avg. number of zeros} & \textnormal{Avg. pattern size} \\
\hline
$R_{0.01}^{0.7}$ & $1$ & $15.915$ ($\times 1$) & $0.062832$ \\
$\mathcal{Q}_{0.01}^1$ & $1.032$ & $16.167$ ($\times 1.016$) & $0.061856$ \\
$\mathcal{Q}_{0.01}^2$ & $1.891$ & $21.887$ ($\times 1.375$) & $0.045690$ \\
$\mathcal{Q}_{0.01}^3$ (hor.) & $1.891$ & $21.887$ ($\times 1.375$) & $0.045690$\\
$\mathcal{Q}_{0.01}^3$ (ver.) & $5.374$ & $36.894$ ($\times 2.318$) & $0.027105$.
\end{tabular}\vspace{4pt}
}
\end{center}
\medskip

If we sample random vertical lines in the numerical simulations represented in Figure~\ref{im:4cases}, 
we get the following number of zeros, which are in good agreement with the predicted asymptotic results.

\begin{center}
{\setlength\extrarowheight{4pt}
\begin{tabular}{l|c}
\textnormal{Domain} & \textnormal{Avg. number of zeros (sampled)} \\
\hline
$R_{0.01}^{0.7}$ & $16.413$ \\
$\mathcal{Q}_{0.01}^1$ & $16.984$ \\
$\mathcal{Q}_{0.01}^2$ & $21.931$ \\
$\mathcal{Q}_{0.01}^3$ (ver.) & $37.315$.
\end{tabular}\vspace{4pt}
}
\end{center}

\section{Numerical Simulations} 
\label{sec:numerical}

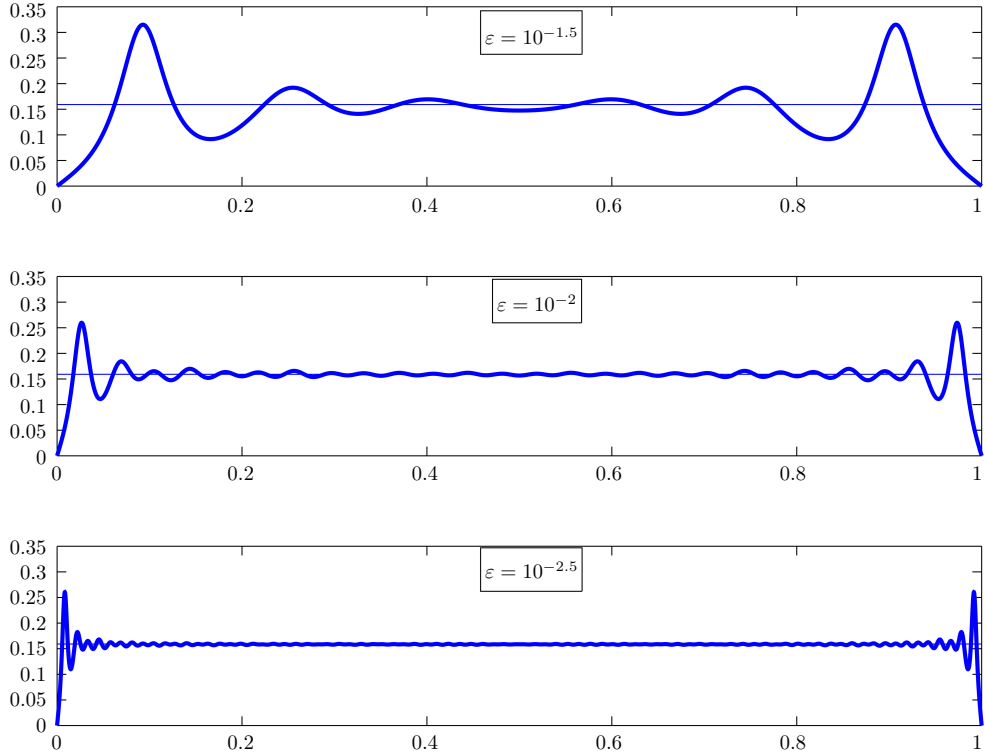
\begin{figure}[htbp!]
\centering
\hspace{-.065\textwidth}\scalebox{.78}{
\setlength{\unitlength}{1pt}
\begin{picture}(0,0)
\includegraphics{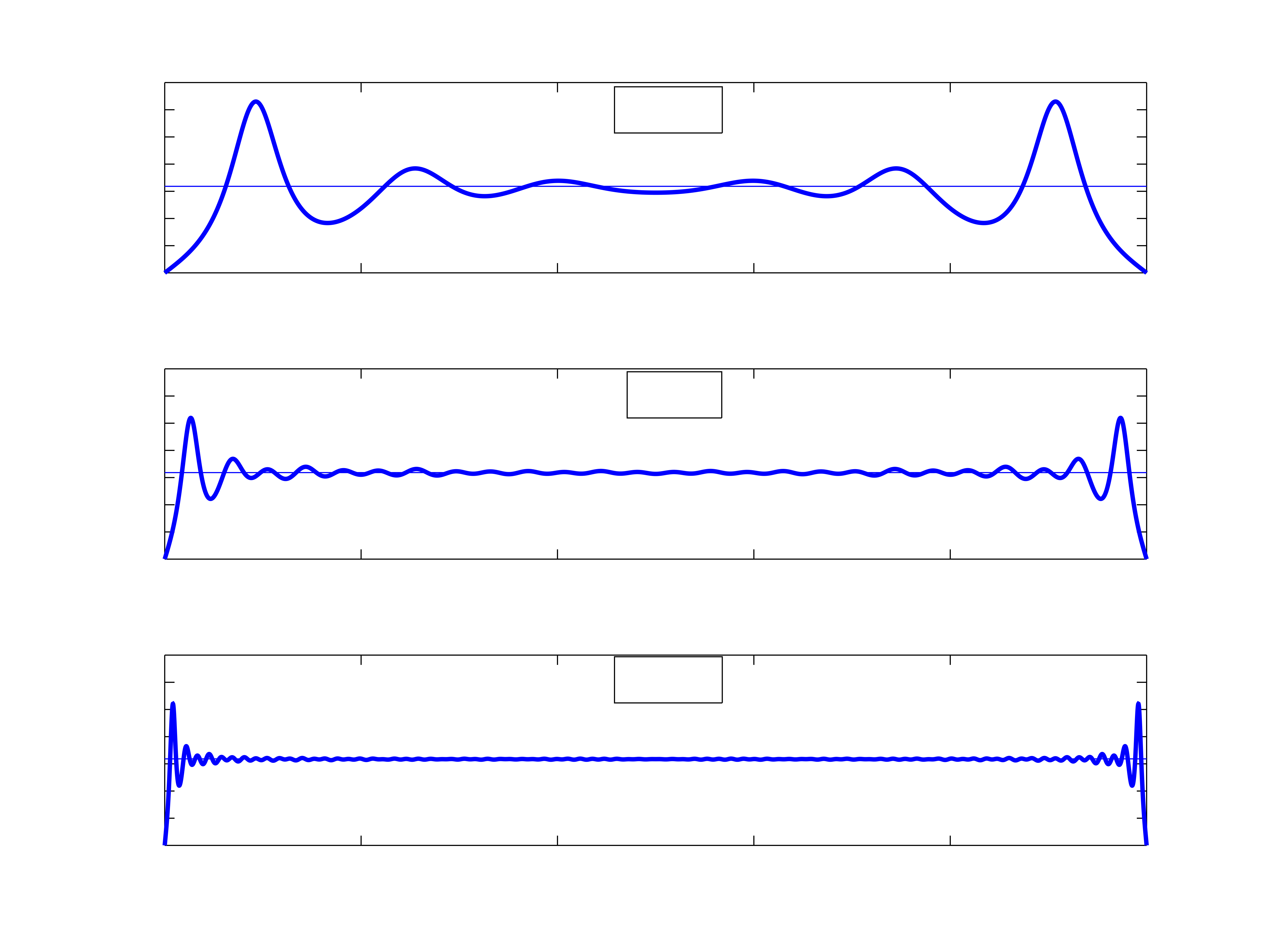}
\end{picture}%
\begin{picture}(576,433)(0,0)
\fontsize{10}{0}
\selectfont\put(74.88,303.934){\makebox(0,0)[t]{\textcolor[rgb]{0,0,0}{{0}}}}
\fontsize{10}{0}
\selectfont\put(164.16,303.934){\makebox(0,0)[t]{\textcolor[rgb]{0,0,0}{{0.2}}}}
\fontsize{10}{0}
\selectfont\put(253.44,303.934){\makebox(0,0)[t]{\textcolor[rgb]{0,0,0}{{0.4}}}}
\fontsize{10}{0}
\selectfont\put(342.72,303.934){\makebox(0,0)[t]{\textcolor[rgb]{0,0,0}{{0.6}}}}
\fontsize{10}{0}
\selectfont\put(432,303.934){\makebox(0,0)[t]{\textcolor[rgb]{0,0,0}{{0.8}}}}
\fontsize{10}{0}
\selectfont\put(521.28,303.934){\makebox(0,0)[t]{\textcolor[rgb]{0,0,0}{{1}}}}
\fontsize{10}{0}
\selectfont\put(69.8755,308.908){\makebox(0,0)[r]{\textcolor[rgb]{0,0,0}{{0}}}}
\fontsize{10}{0}
\selectfont\put(69.8755,321.271){\makebox(0,0)[r]{\textcolor[rgb]{0,0,0}{{0.05}}}}
\fontsize{10}{0}
\selectfont\put(69.8755,333.635){\makebox(0,0)[r]{\textcolor[rgb]{0,0,0}{{0.1}}}}
\fontsize{10}{0}
\selectfont\put(69.8755,345.998){\makebox(0,0)[r]{\textcolor[rgb]{0,0,0}{{0.15}}}}
\fontsize{10}{0}
\selectfont\put(69.8755,358.361){\makebox(0,0)[r]{\textcolor[rgb]{0,0,0}{{0.2}}}}
\fontsize{10}{0}
\selectfont\put(69.8755,370.724){\makebox(0,0)[r]{\textcolor[rgb]{0,0,0}{{0.25}}}}
\fontsize{10}{0}
\selectfont\put(69.8755,383.088){\makebox(0,0)[r]{\textcolor[rgb]{0,0,0}{{0.3}}}}
\fontsize{10}{0}
\selectfont\put(69.8755,395.451){\makebox(0,0)[r]{\textcolor[rgb]{0,0,0}{{0.35}}}}
\fontsize{10}{0}
\selectfont\put(74.88,173.74){\makebox(0,0)[t]{\textcolor[rgb]{0,0,0}{{0}}}}
\fontsize{10}{0}
\selectfont\put(164.16,173.74){\makebox(0,0)[t]{\textcolor[rgb]{0,0,0}{{0.2}}}}
\fontsize{10}{0}
\selectfont\put(253.44,173.74){\makebox(0,0)[t]{\textcolor[rgb]{0,0,0}{{0.4}}}}
\fontsize{10}{0}
\selectfont\put(342.72,173.74){\makebox(0,0)[t]{\textcolor[rgb]{0,0,0}{{0.6}}}}
\fontsize{10}{0}
\selectfont\put(432,173.74){\makebox(0,0)[t]{\textcolor[rgb]{0,0,0}{{0.8}}}}
\fontsize{10}{0}
\selectfont\put(521.28,173.74){\makebox(0,0)[t]{\textcolor[rgb]{0,0,0}{{1}}}}
\fontsize{10}{0}
\selectfont\put(69.8755,178.714){\makebox(0,0)[r]{\textcolor[rgb]{0,0,0}{{0}}}}
\fontsize{10}{0}
\selectfont\put(69.8755,191.077){\makebox(0,0)[r]{\textcolor[rgb]{0,0,0}{{0.05}}}}
\fontsize{10}{0}
\selectfont\put(69.8755,203.441){\makebox(0,0)[r]{\textcolor[rgb]{0,0,0}{{0.1}}}}
\fontsize{10}{0}
\selectfont\put(69.8755,215.804){\makebox(0,0)[r]{\textcolor[rgb]{0,0,0}{{0.15}}}}
\fontsize{10}{0}
\selectfont\put(69.8755,228.167){\makebox(0,0)[r]{\textcolor[rgb]{0,0,0}{{0.2}}}}
\fontsize{10}{0}
\selectfont\put(69.8755,240.53){\makebox(0,0)[r]{\textcolor[rgb]{0,0,0}{{0.25}}}}
\fontsize{10}{0}
\selectfont\put(69.8755,252.894){\makebox(0,0)[r]{\textcolor[rgb]{0,0,0}{{0.3}}}}
\fontsize{10}{0}
\selectfont\put(69.8755,265.257){\makebox(0,0)[r]{\textcolor[rgb]{0,0,0}{{0.35}}}}
\fontsize{10}{0}
\selectfont\put(74.88,43.5463){\makebox(0,0)[t]{\textcolor[rgb]{0,0,0}{{0}}}}
\fontsize{10}{0}
\selectfont\put(164.16,43.5463){\makebox(0,0)[t]{\textcolor[rgb]{0,0,0}{{0.2}}}}
\fontsize{10}{0}
\selectfont\put(253.44,43.5463){\makebox(0,0)[t]{\textcolor[rgb]{0,0,0}{{0.4}}}}
\fontsize{10}{0}
\selectfont\put(342.72,43.5463){\makebox(0,0)[t]{\textcolor[rgb]{0,0,0}{{0.6}}}}
\fontsize{10}{0}
\selectfont\put(432,43.5463){\makebox(0,0)[t]{\textcolor[rgb]{0,0,0}{{0.8}}}}
\fontsize{10}{0}
\selectfont\put(521.28,43.5463){\makebox(0,0)[t]{\textcolor[rgb]{0,0,0}{{1}}}}
\fontsize{10}{0}
\selectfont\put(69.8755,48.52){\makebox(0,0)[r]{\textcolor[rgb]{0,0,0}{{0}}}}
\fontsize{10}{0}
\selectfont\put(69.8755,60.8833){\makebox(0,0)[r]{\textcolor[rgb]{0,0,0}{{0.05}}}}
\fontsize{10}{0}
\selectfont\put(69.8755,73.2466){\makebox(0,0)[r]{\textcolor[rgb]{0,0,0}{{0.1}}}}
\fontsize{10}{0}
\selectfont\put(69.8755,85.6099){\makebox(0,0)[r]{\textcolor[rgb]{0,0,0}{{0.15}}}}
\fontsize{10}{0}
\selectfont\put(69.8755,97.9732){\makebox(0,0)[r]{\textcolor[rgb]{0,0,0}{{0.2}}}}
\fontsize{10}{0}
\selectfont\put(69.8755,110.336){\makebox(0,0)[r]{\textcolor[rgb]{0,0,0}{{0.25}}}}
\fontsize{10}{0}
\selectfont\put(69.8755,122.7){\makebox(0,0)[r]{\textcolor[rgb]{0,0,0}{{0.3}}}}
\fontsize{10}{0}
\selectfont\put(69.8755,135.063){\makebox(0,0)[r]{\textcolor[rgb]{0,0,0}{{0.35}}}}
\fontsize{10}{0}
\selectfont\put(279.36,383.02){\makebox(0,0)[l]{\textcolor[rgb]{0,0,0}{{ $\varepsilon = 10^{-1.5}$ }}}}
\fontsize{10}{0}
\selectfont\put(285.12,253.42){\makebox(0,0)[l]{\textcolor[rgb]{0,0,0}{{ $\varepsilon = 10^{-2}$ }}}}
\fontsize{10}{0}
\selectfont\put(279.36,123.82){\makebox(0,0)[l]{\textcolor[rgb]{0,0,0}{{ $\varepsilon = 10^{-2.5}$ }}}}
\end{picture}
}
\caption{Zeros density for $t=0.5$ and $\gamma=0.8$ (multiplied by $\eps$) for some values of $\eps$ on the ring-shaped domain. These pictures are a graphical representation of the pointwise limit $\eps \delta(x)\to \frac{1}{2\pi}$ proven in Theorem \ref{thm:main}.}
\label{im:deltas}
\end{figure}

In Figure \ref{im:deltas} we plot the (rescaled) density of zeros 
$\eps\delta(x) = \dfrac{\eps}{\pi}\sqrt{W_t(x)}$ 
for various values of $\eps$ on the domain $R^{0.8}$.
We scale by $\eps$ so that the convergence to $\sfrac{1}{2\pi}$ is easily seen. 
We can observe that the density in fact converges pointwise (apart from $x=0,1$). 
It even seems that the convergence is uniform away from an arbitrarily small boundary layer.

It is interesting to fix a point $x_0\in(0,1)$ and to track the value of $\eps\delta(x_0)$ 
for $\eps\to 0$. This is done in Figure \ref{im:followX}. 

\begin{figure}[htbp!]
\centering
\hspace{-.065\textwidth}\scalebox{.78}{
\setlength{\unitlength}{1pt}
\begin{picture}(0,0)
\includegraphics{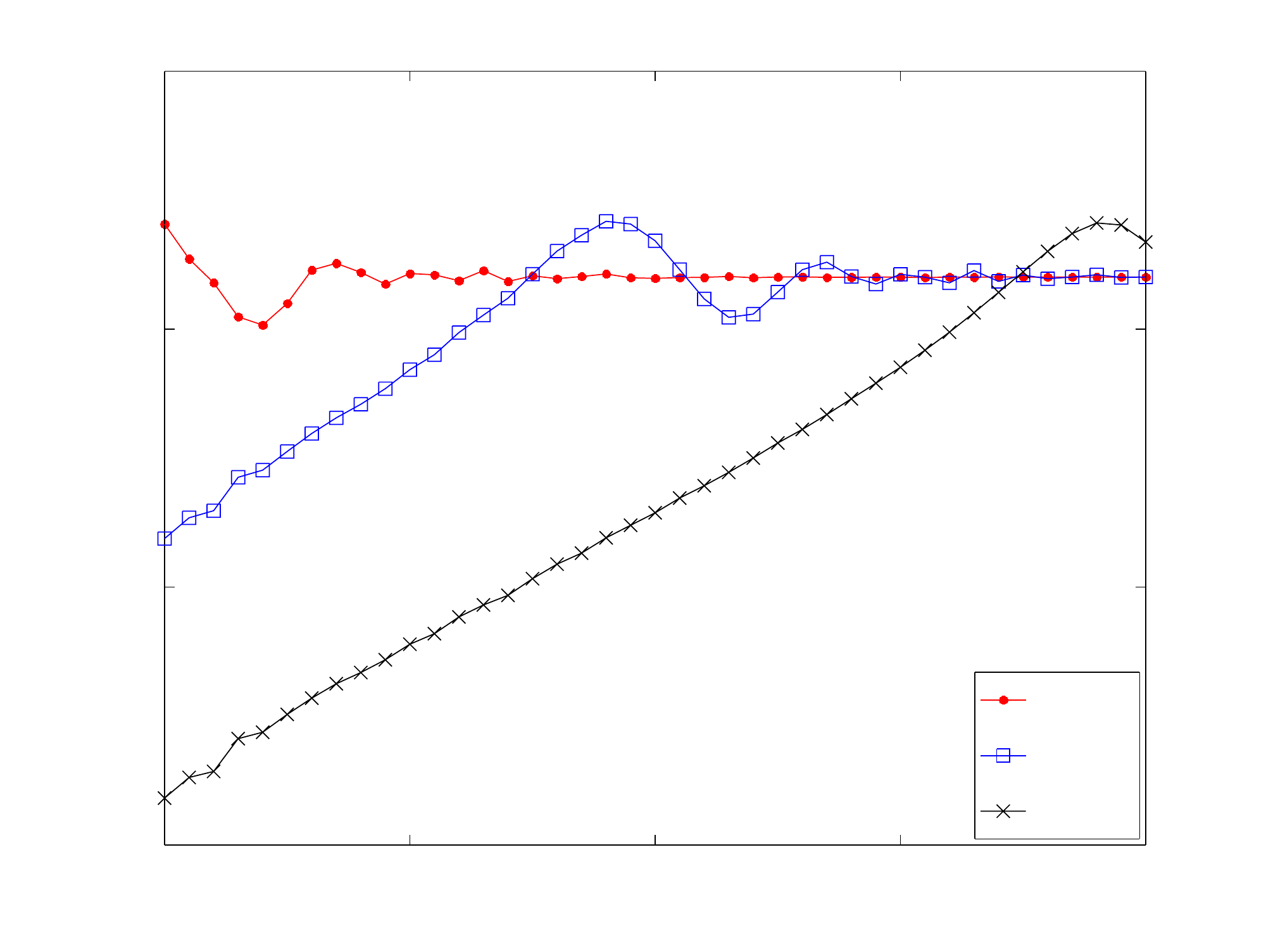}
\end{picture}%
\begin{picture}(576,433)(0,0)
\fontsize{10}{0}
\selectfont\put(74.88,43.5188){\makebox(0,0)[t]{\textcolor[rgb]{0,0,0}{{1.5}}}}
\fontsize{10}{0}
\selectfont\put(186.48,43.5188){\makebox(0,0)[t]{\textcolor[rgb]{0,0,0}{{2}}}}
\fontsize{10}{0}
\selectfont\put(298.08,43.5188){\makebox(0,0)[t]{\textcolor[rgb]{0,0,0}{{2.5}}}}
\fontsize{10}{0}
\selectfont\put(409.68,43.5188){\makebox(0,0)[t]{\textcolor[rgb]{0,0,0}{{3}}}}
\fontsize{10}{0}
\selectfont\put(521.28,43.5188){\makebox(0,0)[t]{\textcolor[rgb]{0,0,0}{{3.5}}}}
\fontsize{10}{0}
\selectfont\put(69.8755,48.52){\makebox(0,0)[r]{\textcolor[rgb]{0,0,0}{{1e-3}}}}
\fontsize{10}{0}
\selectfont\put(69.8755,165.88){\makebox(0,0)[r]{\textcolor[rgb]{0,0,0}{{1e-2}}}}
\fontsize{10}{0}
\selectfont\put(69.8755,283.24){\makebox(0,0)[r]{\textcolor[rgb]{0,0,0}{{1e-1}}}}
\fontsize{10}{0}
\selectfont\put(69.8755,400.6){\makebox(0,0)[r]{\textcolor[rgb]{0,0,0}{{1e+0}}}}
\fontsize{10}{0}
\selectfont\put(298.08,30.5189){\makebox(0,0)[t]{\textcolor[rgb]{0,0,0}{{$n$, where $\varepsilon = 10^{-n}$}}}}
\fontsize{10}{0}
\selectfont\put(39.8755,224.56){\rotatebox{90}{\makebox(0,0)[b]{\textcolor[rgb]{0,0,0}{{$\varepsilon \cdot \delta(x_0)$}}}}}
\fontsize{10}{0}
\selectfont\put(469.504,114.542){\makebox(0,0)[l]{\textcolor[rgb]{0,0,0}{{$\ x_0 = 10^{-1}$}}}}
\fontsize{10}{0}
\selectfont\put(469.504,89.2285){\makebox(0,0)[l]{\textcolor[rgb]{0,0,0}{{$\ x_0 = 10^{-2}$}}}}
\fontsize{10}{0}
\selectfont\put(469.504,63.9152){\makebox(0,0)[l]{\textcolor[rgb]{0,0,0}{{$\ x_0 = 10^{-3}$}}}}
\end{picture}
}
\caption{Values for $\eps\delta(x_0)$ for various values of $x_0$. The parameter $\eps$ is on the abscissa and the scaled value of the density is on the ordinate. }\label{im:followX}
\end{figure}

Once again we can read off the convergence to $\sfrac{1}{2\pi}$ 
but we can also quantify the convergence's disturbance by the ``travelling wave'' 
which goes to the borders of the interval: even if we choose a point which is very close to $0$ (the critical points), 
for example $x_0=10^{-2}$, we see that the magnitude of the density grows very quickly at first until it hits the limit value and then exhibits a damped oscillation around that asymptotic growth rate. 
When compared to Figure \ref{im:deltas}, we can imagine those travelling waves approaching 
and going through such a value $x_0$ until after some threshold $\eps$, 
the oscillation is negligible.

\section{Sloped lines}\label{sec:sloped}

As we have already mentioned, the result given in Theorem~\ref{thm:main} 
is proven only for horizontal lines and by symmetry for vertical lines. 
We claim that the result is far more general.
In order to do so, we   consider  any sloped line through 
the origin. For the ring $R^\gamma$ we will see that we obtain the same asymptotic behaviour.

To fix the setting we assume $y=\mu x$,
with $\mu\in (0,1]$. 
This last assumption is just for simplicity of presentation, 
as we can use the reflection at the diagonal
to get the analogous result for the remaining half square.

Note that considering lines through the origin is not restrictive, 
we do this only for simplicity of presentation. 
The more general case $y=\mu x+\tau$ behaves in the same way, 
using the trigonometric formulas for the cosine of a sum, 
which just adds to the number of the terms involved. 

The important thing in that situation is to consider 
only the intersection $(x, \mu x+\tau)\cap [0,1]^2$, 
so we have to pay attention to the $x$'s that are in our domain, 
and to the corresponding length of the segment we need for the renormalisation.

For $y=\mu x$, we have, instead of $w_t(x)$, the following vector of functions:
\begin{equation*}
w_\mu(x)=\left(\cos(k\pi x)\cos(l\pi\mu x)/\left(\sum_{m,n}\cos^2(m\pi x)\cos^2(n\pi \mu x)\right)^{1/2}\right)_{k,l}
\end{equation*}
and also 
\begin{align*}
W_\mu (x) & = \frac{\sum_{k,l}\left[-S_1(k\pi \sin(k\pi x)\cos(l\pi \mu x)+l\pi \mu \sin(l\pi \mu x)\cos(k\pi x)) + \cos(k \pi x)\cos(l\pi\mu x)\tilde{S}_2\right]^2}{(S_1)^{3}}\\
& = \frac{\sum_{k,l}(k\pi \sin(k\pi x)\cos(l\pi \mu x)+l\pi \mu \sin(l\pi \mu x)\cos(k\pi x))^2}{S_1}-\left(\frac{\tilde{S}_2}{S_1}\right)^2\\
& = \frac{\tilde{S}_3}{S_1}-\left(\frac{\tilde{S}_2}{S_1}\right)^2,
\end{align*}
where $S_1$ is as in the horizontal case  with $t :=\mu x$ and 
\begin{gather*}
\tilde{S}_2 = \sum_{k,l}\cos(k\pi x)\cos(l\pi t)(k\pi \sin(k\pi x)\cos(l\pi t)+l\pi \mu \sin(l\pi t)\cos(k\pi x))\\
\tilde{S}_3 = \sum_{k,l}(k\pi \sin(k\pi x)\cos(l\pi t)+l\pi \mu \sin(l\pi t)\cos(k\pi x))^2,
\end{gather*}
which are the generalisations of the $S_2$ and $S_3$ encountered before.
To prove that things work the same way in this case, we have to prove first the following:
\begin{equation}
\label{e:so}
\lim_{\eps\to0}\eps^2 \left(\frac{\tilde{S}_2}{S_1}\right)^2=0.
\end{equation}
This can be shown with arguments analogous to those for horizontal lines.
The main new tool we need is that the averaging works for weights $(k,1)$, $(1,l)$, and $(k,l)$.
Then we need to notice that any term with a $\sin$-function averages
in the limit to $0$ in the Birkhoff ergodic theorem. This implies~\eqref{e:so}.

Secondly, with the same argument, we see the following asymptotic equivalence
\[
\eps^2 \frac{\tilde{S}_3}{S_1} 
\sim \eps^2 \frac{1}{S_1} \Big( 
\sum_{k,l}k^2\pi^2 \sin(k\pi x)^2\cos(l\pi \mu x)^2
+ \mu^2 \sum_{k,l} l^2\pi^2 \sin(l\pi \mu x)^2 \cos(k\pi x)^2 
\Big) 
\]
This implies the following Theorem:

\begin{theorem}
For sloped lines $y=\mu x$ with $\mu\in[0.1]$ and general Fourier-domains $D$ we have that 
\[
 \delta(x)^2
 \sim 
  \frac{1}{4\eps^2}\cdot \frac1{\lambda(D)} \Big(
\lambda_{(k^2,1)}(D)
+ \mu^2 \lambda_{(1,l^2)}(D)
\Big).
 \]
\end{theorem}

If the domain $D$ is symmetric with respect to the reflection along the diagonal
like the ring $R^\gamma$ or the squares $\mathcal{Q}_1$ and $\mathcal{Q}_2$. 
Then  $\delta(x)^2$ is compared to a horizontal line just modified by a factor $(1+\mu^2)$.
But this just compensates the length of the segment, 
which is not 1 as before but $\sqrt{1+\mu^2}$.
Now the two factors cancel and the average pattern size is
for every sloped line the same as for the horizontal line.

Thus we obtain in the symmetric examples like $R^\gamma$, $\mathcal{Q}_1$, and $\mathcal{Q}_2$
that for any sloped line the  average pattern size is
for every sloped line the same and of order $\eps$.

\appendix
\section{The rational case}\label{sec:rational}
In the proof of Lemma \ref{lem:maincorollary} we only considered irrational values of $x$ and $t$. This is sufficient because we are interested in the density on a set of Lebesgue-measure 1 because we integrate it anyway. But as an additional consideration, for rational values, the periodicity of trigonometric functions yields a non-ergodic averaging result of the same type as on the irrational numbers.

We restrict ourselves to a particularly easy case but the actual proof is a straightforward generalization.

Let $x = \frac{1}{n}$ and $N = l\cdot n$ for $l,n\in\N$ without loss of generalization and consider
\begin{align*}
\frac{1}{N}\sum_{k=1}^N\cos^2(k\pi x) &= \frac{1}{2N}\sum_{k=1}^N\cos(2\pi k\frac{1}{n})\\
&= \frac{1}{2}+ \frac{1}{2}\cdot \left(\frac{1}{2}\frac{\sin(\frac{2\pi N+\pi}{n})-1}{\sin(\frac{\pi}{n})}\right)\\
&=\frac{1}{2}-\frac{1}{4N}\xrightarrow{N\to\infty}\frac{1}{2}.
\end{align*}
This corresponds nicely to the ergodic result
\[
\lim_{N\to\infty}\frac{1}{N}\sum_{k=1}^N\cos^2(k\pi x) = \frac{1}{2}
\]
for irrational $x$. In this way, we can for example prove 
\[
\frac1{|R_\eps|}S_1 = \frac{1}{|R_\eps|}\sum_{k,l\in R_\eps}\cos^2(k\pi x)\cos^2(l\pi t)\xrightarrow{N\to\infty}\frac{1}{4}.
\]

\section*{Acknowledgments}
 The authors want to thank Thomas Wanner for pointing out the result by Edelman and Kostlan.
 Moreover, Blömker \& Wacker would like to thank Evelyn Sander and Thomas Wanner for their hospitality.

\bibliographystyle{abbrv}
\bibliography{bib}
\end{document}